\documentclass{IEEEtran}
\usepackage{ifthen}

\newboolean{tacVersion}
\setboolean{tacVersion}{false}

\ifthenelse{\boolean{tacVersion}}{
    \newcommand{\revise}[1]{#1}
    \newcommand{\revisetwo}[1]{#1}
    \newcommand{\revisethree}[1]{#1}
}{
    \newcommand{\revise}[1]{#1}
    \newcommand{\revisetwo}[1]{#1}
    \newcommand{\revisethree}[1]{#1}
}

\usepackage{textcomp}
\def\BibTeX{{\rm B\kern-.05em{\sc i\kern-.025em b}\kern-.08em
    T\kern-.1667em\lower.7ex\hbox{E}\kern-.125emX}}

\usepackage{amsmath} %
\usepackage{amssymb}  %

\usepackage{amsthm}
\usepackage{algorithm}
\usepackage[noend]{algpseudocode}
\usepackage{caption}
\usepackage{subcaption}

\usepackage{enumitem}
\usepackage{booktabs}
\usepackage{threeparttable}

\usepackage{tikz}
\usepackage{pgfplots}
\usetikzlibrary{external}
\DeclareMathOperator{\tr}{tr}

\newtheorem{theorem}{Theorem}
\newtheorem{proposition}[theorem]{Proposition}
\newtheorem{lemma}[theorem]{Lemma}
\newtheorem{corollary}[theorem]{Corollary}
\theoremstyle{definition}
\newtheorem{definition}[theorem]{Definition}

\newtheorem{remark}[theorem]{Remark}

\usetikzlibrary{intersections,fillbetween,backgrounds}

\newcommand{\equalsim}{\mathrel{\stackrel{\sim}{=}}}

\usepackage{cite}

\makeatletter
\let\NAT@parse\undefined
\makeatother
\usepackage{hyperref}

\title{Almost Surely \(\sqrt{T} \) Regret for Adaptive LQR}

\author{Yiwen Lu and Yilin Mo%
\thanks{
The authors are with the Department of Automation and BNRist, Tsinghua University, Beijing, P.R.China. Emails:  luyw20@mails.tsinghua.edu.cn, ylmo@tsinghua.edu.cn.

This work is supported by the National Natural Science Foundation of China under Grant 62461160313, 62192752, 62273196, the BNRist project (No. BNR2024TD03003), and the 111 International Collaboration Project (No. B25027).
}}

\begin{document}

\maketitle
\thispagestyle{empty}
\pagestyle{empty}

\begin{abstract}
  The Linear-Quadratic Regulation (LQR) problem with unknown system parameters has been widely studied, but it has remained unclear whether \( \tilde{ \mathcal{O}}(\sqrt{T}) \) regret, which is the best known dependence on time, can be achieved almost surely. In this paper, we propose an adaptive LQR controller with almost surely \( \tilde{ \mathcal{O}}(\sqrt{T}) \) regret upper bound.
  \revise{
  The controller features a circuit-breaking mechanism, which falls back to a known stabilizing controller when the input signal or control gain is large, thereby circumventing potential safety breach and ensuring the convergence of the system parameter estimate. Meanwhile the circuit-breaking is shown to be triggered only finitely often, and hence has a negligible effect on the asymptotic performance of the controller.}
  The proposed controller is also validated via simulation on Tennessee Eastman Process~(TEP), a commonly used industrial process example.
\end{abstract}

\section{Introduction}

Adaptive control, the study of decision-making under the parametric uncertainty of dynamical systems, has been pursued for decades\revise{~\cite{aastrom1973self,morse1980global,lai1986extended,kuipers2010multiple,zhu2015adaptive,hou2019model}}.
Although early research mainly focused on the aspect of convergence and stability,
recent years have witnessed significant advances in the quantitative performance analysis of adaptive controllers, especially for multivariate systems.
In particular, in the adaptive Linear-Quadratic Regulation~(LQR) setting considered in this paper, the controller attempts to solve the stochastic LQR problem without access to the true system parameters, and its performance is evaluated via \emph{regret}, the cumulative deviation from the optimal cost over time.
\revisetwo{Recent studies have extensively focused on this adaptive LQR setting}~\cite{abbasi2011online,abeille2018improved,cohen2019learning,dean2018regret,simchowitz2020naive,faradonbeh2020adaptive,wang2021exact}, but it has remained unclear whether adaptive controllers for LQR can achieve \( \tilde{\mathcal{O}}(\sqrt{T}) \) regret, whose asymptotic dependence on the time \( T \) is the best known (up to poly-logarithmic factors)~\footnote{It is known that \( \sqrt{T} \) is the optimal rate of \revise{expected regret~\cite{simchowitz2020naive}, i.e., $\mathbb{E}[\text{Regret}(T)] = \Theta(\sqrt T)$}, and we suspect that it is also the optimal almost sure rate.}, \emph{almost surely}.

\revisetwo{
As summarized in Table~\ref{tbl:review}, existing regret upper bounds on adaptive LQR often provide probabilistic guarantees that are weaker than almost-sure or exhibit suboptimal asymptotic time dependence.
Strategies such as optimism-in-face-of-uncertainty~\cite{cohen2019learning}, Thompson sampling~\cite{abeille2018improved}, and the \( \epsilon \)-greedy algorithm~\cite{simchowitz2020naive} can achieve \( \tilde{\mathcal{O}}(\sqrt{T}) \) regret with a probability of at least \( 1 - \delta \).} In other words, these algorithms may not converge to the optimal one or could even be destablizing with a non-zero probability $\delta$.
In practice, such a failure probability may hinder the application of these algorithms in safety-critical scenarios.
From a theoretical perspective, we argue that it is highly difficult to extend these algorithms to provide stronger performance guarantees. To be specific, despite different exploration strategies, the aforementioned methods all compute the control input using a linear feedback gain synthesized from a least-squares estimate of system parameters. Due to Gaussian process noise, the derived linear feedback gain may be destabilizing, regardless of the amount of data collected. For these algorithms, the probability of such catastrophic event is bounded by a positive $\delta>0$, which is a predetermined design parameter and \revisetwo{remains fixed} during online operation.

Alternative methods that may preclude the above described failure probability include introducing additional stability assumptions, using parameter estimation algorithms with stronger convergence guarantees, and adding a layer of safeguard around the linear feedback gain.
Faradonbeh et al.~\cite{faradonbeh2020adaptive} \revise{achieve} \( \tilde{\mathcal{O}}(\sqrt{T}) \) regret almost surely under the assumption that the closed-loop system remains stable all the time, based on a stabilization set obtained from adaptive stabilization~\cite{faradonbeh2018finite}. \revise{Given that the stabilization set is derived from finite data and breaches the desired property with a non-zero probability, the method inherently possesses a non-zero probability of failure.}
Guo~\cite{guo1996self} achieves sub-linear regret almost surely by adopting a variant of ordinary least squares with annealing weight assigned to recent data, a parameter estimation algorithm convergent even with unstable trajectory data. However, the stronger convergence guarantee may come at the cost of less sharp asymptotic rate, and it is unclear whether the regret of this method can achieve \( \tilde{\mathcal{O}}(\sqrt{T}) \) dependence on time.
Wang and Janson~\cite{wang2021exact} \revise{achieve} \( \tilde{\mathcal{O}}(\sqrt{T}) \) regret with a convergence-in-probability guarantee, via the use of a switched, rather than linear feedback controller, which falls back to a known stabilizing gain on the detection of large states. However, the controller design of \revise{that} work does not rule out the frequent switching between the learned and fallback gains, a typical source of instability in switched linear systems~\cite{lin2009stability}, which restricts the correctness of their results to the case of commutative closed-loop system matrices. Moreover, the regret analysis in that work is not sufficiently refined to lead to almost sure guarantees.

\begin{table}[!htbp]
    \centering
    \begin{threeparttable}
        \caption{Comparison with selected existing works on adaptive LQR}
        \label{tbl:review}
        \begin{tabular}{@{}lll@{}}
        \toprule
        \bf{Work} & \bf{Rate} & \bf{Type of guarantee}            \\ \midrule
        \cite{guo1996self}    &    Not provided   & Almost sure                  \\
        \cite{cohen2019learning,abeille2018improved,simchowitz2020naive}     &  \( \tilde{ \mathcal{O}}(\sqrt{T}) \)    & Probability \( 1 - \delta \) \\
        \cite{faradonbeh2020adaptive}     &   \( \tilde{ \mathcal{O}}(\sqrt{T}) \)    & Almost sure\tnote{*}                  \\
        \cite{wang2021exact}    &   \( \tilde{ \mathcal{O}}(\sqrt{T}) \)   & Convergence in probability\tnote{*}   \\
        \bf{This work}    &   \( \tilde{ \mathcal{O}}(\sqrt{T}) \)   & Almost sure                  \\ \bottomrule
        \end{tabular}
        \begin{tablenotes}
            \item[*] Requires non-standard assumptions; please refer to the \revise{body} text for details.
        \end{tablenotes}
    \end{threeparttable}
\end{table}

In this paper, we present an adaptive LQR controller with \( \tilde{\mathcal{O}}(\sqrt{T}) \) regret almost surely, only assuming the availability of a known stabilizing feedback gain, which is a common assumption in the literature.
This is achieved by a ``circuit-breaking'' mechanism motivated similarly as~\cite{wang2021exact}, which circumvents safety breach by supervising the norm of the state and deploying the feedback gain when necessary. In contrast to~\cite{wang2021exact}, however, by enforcing a properly chosen dwell time on the fallback mode, the stability of the closed-loop system under our proposed controller is unaffected by switching.
Another insight underlying our analysis is that the above mentioned circuit-breaking mechanism is triggered only finitely often. \revise{This observation suggests that the conservativeness of the proposed controller, while preventing early-stage destabilization and ensuring convergence of system parameter estimates, has a negligible impact on the system's asymptotic performance.} Although similar phenomena have also been observed in~\cite{lei1988convergence,wang2021exact}, we derive an upper bound on the time of the last trigger (Theorem~\ref{thm:prop}, item~\ref{item:Tnocb})), a property missing from pervious works that paves the way to our almost sure regret guarantee.

The remainder of this manuscript is organized as follows: Section~\ref{sec:problem} introduces the problem setting. Section~\ref{sec:controller} describes the proposed controller. Section~\ref{sec:theory} states and proves the theoretical properties of the closed-loop system under the proposed controller, and establishes the main regret upper bound.
Section~\ref{sec:simulation} validates the theoretical results using a numerical example. Finally, Section~\ref{sec:conclusion} summarizes the manuscript and discusses directions for future work.

\subsection*{Notations}

The set of nonnegative integers is denoted by $\mathbb{N}$, and the set of positive integers is denoted by $\mathbb{N}^*$.
The $n$-dimensional Euclidean space is denoted by $\mathbb{R}^n$, and the $n$-dimensional unit sphere is denoted by $\mathbb{S}^n$.
The \( n\times n \) identity matrix is denoted by \( I_n \).
For a square matrix $M$, $\rho(M)$ denotes the spectral radius of $M$, and $\tr(M)$ denotes the trace of $M$.
For a real symmetric matrix $M$, $M\succ 0$ denotes that $M$ is positive definite.
For any matrix $M$, $M^\dag$ denotes the Moore-Penrose inverse of $M$.
For two vectors $u, v\in \mathbb{R}^n$, $\langle u, v \rangle$ denotes their inner product.
For a vector $v$, $\|v\|$ denotes its 2-norm, and $\| v \|_P = \| P^{1/2} v \|$ for $P \succ 0$.
For a matrix $M$, $\|M\|$ denotes its induced 2-norm, and $\| M \|_F$ denotes its Frobenius norm.
For a random vector $x$, $x\sim \mathcal{N}(\mu,\Sigma)$ denotes $x$ is Gaussian distributed with mean $\mu$ and covariance $\Sigma$.
For a random variable $X$, $X \sim \chi^2(n)$ denotes $X$ has a chi-squared distribution with $n$ degrees of freedom.
$\mathbb{P}(\cdot)$ denotes the probability operator, and $\mathbb{E}[\cdot]$ denotes the expectation operator. \revise{For a random event $\mathcal{E}$, $\bar{\mathcal{E}}$ denotes the complement of $\mathcal{E}$.}

For non-negative quantities $f,g$, which can be deterministic or random, we say $f \lesssim g$ to denote that $f \leq C_1 g + C_2$ for some universal \revise{(possibly system-dependent)} constants $C_1 > 0$ and $C_2 > 0$, and $f \gtrsim g$ to denote that $g \lesssim f$.
For a random function $f(T)$ and a deterministic function $g(T)$, we say $f(T) = \mathcal{O}(g(T))$ to denote that $\limsup_{T\to\infty} f(T) / g(T) < \infty$, and $f(T) = \tilde{\mathcal{O}}(g(T))$ to denote that $f(T) = \mathcal{O}(g(T) (\log(T))^\alpha)$ for some $\alpha > 0$.

\section{Problem formulation}
\label{sec:problem}

This paper considers a fully observed discrete-time linear system \revisetwo{subject to} Gaussian process noise\revisetwo{, defined by the following equation}:
\begin{equation}
    x_{k+1} = A x_k + B u_k + w_k,\quad k \in \mathbb{N}^*, \quad x_1 = 0,
    \label{eq:dyn}
\end{equation}
where $x_k \in \mathbb{R}^n$ is the state, $u_k \in \mathbb{R}^m$ is the control input, and $w_k \stackrel{\text { i.i.d. }}{\sim} \mathcal{N}(0, W)$  is the process noise, where \( W \succeq 0 \).
It is assumed without loss of generality that \( W = I_n \), but all the conclusions apply to general positive semidefinite \( W \) up to the scaling of constants.
It is also assumed that the system and input matrices $A, B$ are unknown to the controller, but $(A,B)$ is controllable.

\revisetwo{We evaluate the system performance using an average infinite-horizon quadratic cost, which reflects the long-term cost effectiveness of control strategies:}
\begin{equation}
    J=\limsup _{T \rightarrow \infty} \frac{1}{T} \mathbb{E}\left[\sum_{k=1}^{T} x_k^\top Q x_k + u_k^\top R u_k \right],
    \label{eq:J}
\end{equation}
where $Q \succ 0, R \succ 0$ are fixed weighting matrices specified by the system operator.

It is well known that the optimal control law is the linear feedback control law of the form $u(x) = K^* x$, where the optimal feedback gain $K^*$ can be specified as:
\begin{equation}
    K^*=-\left(R+B^{\top} P^* B\right)^{-1} B^{\top} P^* A,
    \label{eq:K}
\end{equation}
and $P^*$ is the positive definite solution to the discrete algebraic Riccati equation:
\begin{equation}
    P^*=A^\top P^* A-A^\top P^* B\left(R+B^\top P^* B\right)^{-1}B^\top P^* A+Q.
    \label{eq:dare}
\end{equation}
The corresponding optimal cost is
\begin{equation}
    J^* = \tr\left( \mathbb{E}\left[w_kw_k^\top\right] P^*\right) = \tr(WP^*) = \tr(P^*).
    \label{eq:Jstar}
\end{equation}
The matrix \(P^*\) also satisfies the discrete Lyapunov equation
\begin{equation}
    (A+BK^*)^\top P^* (A+BK^*) - P^* + Q + K^\top R K = 0,
    \label{eq:lyap_opt}
\end{equation}
which implies that there exists a scalar $0 < \rho^* < 1$ such that
\begin{equation}
    (A+BK^*)^\top P^* (A+BK^*) \prec \rho^* P^*.
    \label{eq:rho_star}
\end{equation}

\revise{
We assume that the system is open-loop stable, i.e., $\rho(A) < 1$.
Consequently, there exists $P_0 \succ 0$ that satisfies the discrete Lyapunov equation
\begin{equation}
    A^\top P_0 A - P_0 + Q = 0,
    \label{eq:P0}
\end{equation}
and there exists a scalar $0 < \rho_0 < 1$ such that
\begin{equation}
    A^\top P_0 A \prec \rho_0 P_0.
    \label{eq:rho0}
\end{equation}
}

\begin{remark}
    It has been commonly assumed in the literature~\cite{simchowitz2020naive,wang2021exact} that \( (A, B) \) is stabilizable by a known feedback gain \( K_0 \), i.e., \( \rho(A+BK_0) < 1 \).
    We assume without loss of generality that the system is open-loop stable, i.e., \( K_0 = 0 \), for the simplicity of notations.
    \revise{However, all the results in this paper can be extended to $K_0 \neq 0$, which we discuss in Section~\ref{sec:unstable}.}
    \label{rem:prestabilizer}
\end{remark}

Since $A,B$ are unknown to the controller in the considered setting, it is not possible to directly compute the optimal control law from~\eqref{eq:K} and~\eqref{eq:dare}. Instead, the controller learns the optimal control law online, whose performance is measured via the \emph{regret} defined as follows:
\begin{equation}
    \mathcal{R}(T) = \sum_{k=1}^T (x_k^\top Qx_k + u_k^\top Ru_k)- TJ^*.
    \label{eq:regret}
\end{equation}
The goal of the controller is to minimize the asymptotic growth of $\mathcal{R}(T)$.

\section{Controller design}
\label{sec:controller}

\revisetwo{
Algorithm~\ref{alg:main} outlines the proposed controller, and Fig.~\ref{fig:block} provides a visual illustration. This section explains each component of the controller.
}

\begin{algorithm}[!htbp]
    \begin{algorithmic}[1]
        \State $\hat{K}_0 \gets 0$
        \State $\xi \gets 0$
        \For{$k=1,2,\ldots$}
            \State Update parameter estimates $\hat{A}_k, \hat{B}_k$ using~\eqref{eq:ols}
            \If{$(\hat{A}_k, \hat{B}_k)$ is \revise{stabilizable}} \label{line:controllable}
                \State Update $\hat{K}_k$ from~\eqref{eq:K}-\eqref{eq:dare}, replacing $(A,B)$ with $(\hat{A}_k, \hat{B}_k)$. \label{line:update}
            \Else
                \State $\hat{K}_k \gets 0$ \label{line:uncontrollable}
            \EndIf
            \State $u^{ce}_k \gets \hat{K}_k x_k$
            \If{$\xi = 0$} \label{line:switching_start}
                \If { $\max\{\| u^{ce}_k \|, \| \hat{K}_k \| \} > M_k := \log(k)$ }
                    \State $\xi \gets t_k := \lfloor \log(k) \rfloor$
                    \label{line:tk}
                    \State $u^{cb}_k \gets 0$
                \Else \label{line:switching_else}
                    \State $u^{cb}_k \gets u^{ce}_k$
                \EndIf
            \Else
                \State $u^{cb}_k \gets 0$
                \State $\xi \gets \xi - 1$ \label{line:switching_end}
            \EndIf
            \State $u^{pr}_k \gets k^{-1/4}v_k$, where $v_k \sim \mathcal{N}(0,I_m)$
            \State Apply $u_k \gets u^{cb}_k + u^{pr}_k$
        \EndFor
    \end{algorithmic}
    \caption{Proposed controller}
    \label{alg:main}
\end{algorithm}

\tikzexternaldisable
\begin{figure}[!htbp]
  \tikzsetnextfilename{figures/block_diagram.tex}%
  \usetikzlibrary{shapes.misc,positioning,calc,graphs,arrows.meta,quotes,fit,backgrounds}
\tikzset{pin/.style={
    circle,
    minimum size=3mm,
    thick, draw,
    inner sep=0
}}
\tikzset{block/.style={
    rectangle,
    thick, draw,
    inner sep=1ex, align=center
}}
\tikzset{wide block/.style={
    block,
    text width=10ex
}}
\tikzset{ultra wide block/.style={
    block,
    text width=14ex
}}
\tikzset{skip/.style={
    to path={-- ++(0, #1) -| (\tikztotarget)}
}}
\tikzset{reverse_S/.style={
    to path={-- ++(#1, 0) |- (\tikztotarget)}
}}
\tikzset{vh/.style={
    to path={|- (\tikztotarget)}
}}
\tikzset{hv/.style={
    to path={-| (\tikztotarget)}
}}
\def\centerarc[#1](#2)(#3:#4:#5)%
    { \draw[#1] ($(#2)+({#5*cos(#3)},{#5*sin(#3)})$) arc (#3:#4:#5); }

\begin{tikzpicture}[scale=0.8,>=latex, every node/.style={scale=0.85}]
    \matrix[row sep=6, column sep=-5] {
    & & \node[ultra wide block] (logic) {Circuit-breaking logic \revise{(Line~\ref{line:switching_start}-\ref{line:switching_end} of Algorithm~\ref{alg:main})}}; & & \\
    & \node[pin] (pin_K0) {} node [above right=2mm] {$\mkern40mu$} node [left=8mm] (zero) {};  & &\\
    \node[block] (K1) {$\hat{K}$}; & \node[pin] (pin_K1) {} node [below=2mm] {$u^{ce}$}; & \node[pin] (pin_mid) {} node [below=2mm] {$u^{cb}$}; & \node[pin] (pin_u) {+} ; & \node[wide block] (plant) {Plant (eq.~\eqref{eq:dyn})};\\
    &  & & \node[wide block] (est) {Estimator (eq.~\eqref{eq:ols})}; & \\
    };

    \coordinate (K1 west) at ($(K1) - (1, 0)$);
    \coordinate (noise) at ($(plant) + (0, 2)$);
    \coordinate (probe) at ($(pin_u) + (0, 2)$);
    \coordinate (next state) at ($ (plant.east) + (1, 0) $);
    \coordinate (feedback point) at ($ (plant.east)!0.5!(next state) $);
    \coordinate (switch arc start) at ($(pin_K1)!0.5!(pin_mid)$);
    \path let \p{1}=(pin_K1), \p{2}=(pin_mid), \n{x dist}={abs(\x{2}-\x{1})}, \p{3}=($(pin_mid)-(pin_K0)$), \n{arc angle}={180+asin(\y{3} / \x{3})} in (pin_K1) arc (180:160:\n{x dist}) node (switch start) {};
    \draw [->] let \p{1}=(switch arc start), \p{2}=(pin_mid), \p{3}=(pin_K1), \p{4}=(pin_K0), \n{x dist}={abs(\x{2}-\x{1})}, \n{arc angle}={180-asin((\y{4} - \y{3}) /(\x{2} - \x{1}))} in (switch arc start) node [below] { \( \xi = 0 \)} arc (180:135:\n{x dist}) node [above] { \( \xi > 0 \)};

    \coordinate (xi_end_1) at ($ (pin_mid) + (0, 0.7) $);
    \coordinate (xi_end_2) at ($ (xi_end_1) + (-0.3, -0.3) $);
    \graph [use existing nodes, edges=rounded corners] {
        K1 -> pin_K1,
        logic --["$\xi$"] xi_end_1 -> xi_end_2,
        noise ->["$w$"] plant,
        probe ->["$u^{pr}$"] pin_u,
        pin_mid -> pin_u ->["$u$"] plant ->["$x$"] next state,
        feedback point ->[vh] est,
        feedback point --[skip=-3cm] K1 west ->[vh] logic.170,
        K1 west -> K1.west,
        K1 ->[vh] logic.190,
        pin_u -> est,
        est ->[hv] K1,
        switch start -- pin_mid,
        zero ->["0"] pin_K0,
    };
\end{tikzpicture}%

    \caption{Block diagram of the closed-loop system under the proposed controller. The control input is the superimposition of a deterministic input $u^{cb}$ and a random probing input $u^{pr}$. The deterministic part $u^{cb}$ is normally the same as the certainty equivalent input $u^{ce}$, but takes the value zero when circuit-breaking is triggered, where $\xi$ is a counter for circuit-breaking. The certainty equivalent gain is updated using the parameter estimator in the meantime.}
    \label{fig:block}
\end{figure}
\tikzexternalenable

The proposed controller is a variant of the certainty equivalent controller~\cite{mania2019certainty}, where the latter applies the input $u^{ce}_k = \hat{K}_k x_k$, \revise{with the feedback gain $\hat{K}_k$ being calculated} from~\eqref{eq:K}-\eqref{eq:dare} by treating the current estimates of the system parameters $\hat{A}_k, \hat{B}_k$ as the true values.
The proposed controller differs from the standard certainty equivalent controller by i) including a ``circuit-breaking'' mechanism that replaces $u^{ce}_k$ with zero under certain conditions; ii) superimposing a probing noise $u^{pr}_k$ on the control input at each step.

\revisetwo{
The circuit-breaking mechanism sets $u^{ce}_k$ to zero for the subsequent $t_k$ steps if the norm of either the certainty equivalent control input $\| u^{ce}_k \|$ or the feedback gain $\left\| \hat{K}_k \right\|$ exceeds a threshold $M_k$.}
The intuition behind this design is that a large certainty equivalent control input is indicative of having applied a destabilizing feedback gain recently, and circuit-breaking may prevent the state from exploding by leveraging the innate stability of the system, and hence help with the convergence of the parameter estimator and the asymptotic performance of the controller.
The threshold $M_k$ \revisetwo{increases with the time index $k$, reducing the conservativeness of circuit-breaking as the parameter estimates become more accurate.}
The time $t_k$, \revise{known as dwell time in switching control~\cite{ishii2001stabilizing}}, is also increased with time, in order to circumvent the potential oscillation of the system caused by the frequent switching between $\hat{K}$ and $0$ \revise{(see~\cite[Appendix C]{lu2023almost} for an illustrative example)}.
Both $M_k$ and $t_k$ are chosen to grow logarithmically with $k$, which would support our technical guarantees.

Following the approach in~\cite{simchowitz2020naive,wang2021exact}, probing noise $u_k^{pr}$ is added to the control input at each step to ensure sufficient excitation to the system, necessary for accurate parameter estimation. Specifically, the probing noise is chosen to be $u_k^{pr} = k^{-1/4}v_k$, where $v_k \stackrel{\text { i.i.d. }}{\sim}  \mathcal{N}(0,I_m)$, which would correspond to the optimal rate of regret.

The estimates of the system parameters $\hat{A}_k, \hat{B}_k$ are updated using an ordinary least squares estimator. Denote $\Theta = [A\quad B]$, $\hat{\Theta}_k = [\hat{A}_k \quad \hat{B}_k]$, and $z_k = [x_k^\top \quad u_k^\top]^\top$, then according to \revise{$x_{k+1} = \Theta z_k+w_k$}, the Ordinary Least Squares~(OLS) estimator can be specified as:
\begin{equation}
    \hat{\Theta}_k = \left(\sum_{t=1}^{k-2} x_{t+1} z_{t}^\top\right)\left(\sum_{t=1}^{k-2} z_{t} z_{t}^\top\right)^\dag.
    \label{eq:ols}
\end{equation}

\begin{remark}
    \revise{
    Since the OLS estimator~\eqref{eq:ols} can be updated incrementally in constant time per step as the total time steps $k$ increases, using the Sherman-Morrison formula~\cite{kailath2000linear} or incremental QR decomposition~\cite{hammarling2008updating}, the time complexity of Algorithm~\ref{alg:main} is $\mathcal{O}(1)$, i.e., the computation load does not grow as the number of time steps increases.
    }
    \label{rem:complexity}
\end{remark}

\begin{remark}
    In the presented algorithm, the certainty equivalent gain is updated at every step \( k \) (see line~\ref{line:update} of Algorithm~\ref{alg:main}), but it may also be updated ``logarithmically often''~\cite{wang2021exact}, e.g., at steps \( k = 2^i, i\in \mathbb{N}^* \), which is more computationally efficient in practice. It can be verified that all our theoretical results also apply to the case of episodic updates.
    \label{rem:log_often}
\end{remark}

\section{Theoretical results}
\label{sec:theory}

\subsection{Main results}

Our analysis of the closed-loop system relies on two key random times, defined as follows:

\begin{align}
    & T_{\text{stab}} := \inf \left\{ T \vphantom{\left(A^{\lfloor \log k \rfloor}\right)^\top} \right. \left|   \left(A^{t_k}\right)^\top P^* A^{t_k} < \rho P^*, \right.\nonumber \\
    & \mkern10mu \left. \vphantom{\left(A^{t_k}\right)^\top} \left(A+B\hat{K}_k\right)^\top P^* \left(A+B\hat{K}_k\right) < \rho P^*,  \forall k \geq T \right\},
    \label{eq:Tstab}
\end{align}
where \( \rho = (1 + \rho^*) / 2 \), and \( t_k \) is the dwell time defined in line~\ref{line:tk} of Algorithm~\ref{alg:main}. \revisetwo{Additionally,}
\begin{equation}
    T_{\text{nocb}} := \inf \left\{ \revise{T \mid T \geq T_{\text{stab}} \text{ and }} u_k^{cb} \revise{=} u_k^{ce}, \forall k\geq T \right\},
    \label{eq:Tnocb}
\end{equation}
With the above two random times, the evolution of the system can be divided into three stages:
\begin{enumerate}
    \item From the beginning until \( T_{\text{stab}} \), the adaptive controller refines its estimate of system parameters, improving the performance of the certainty equivalent feedback gain. This process continues until the system stabilizes, with common Lyapunov function between the two modes under circuit-breaking as indicated in~\eqref{eq:Tstab}.
    \item From \( T_{\text{stab}} \) to \( T_{\text{nocb}} \), the closed-loop system is stable as is ensured by the aforementioned common Lyapunov function, and under mild regularity conditions on the noise, an upper bound on the magnitude \revise{of the} certainty equivalent control input \( \| u^{ce}_k \| \) eventually drops below the circuit-breaking threshold \( M_k \).
    \item From \( T_{\text{nocb}} \) on, circuit-breaking is not triggered any more, and the system behaves similarly as the closed-loop system under the optimal controller, with only small perturbations on the feedback gain which stem from the parameter estimation error and converge to zero.
\end{enumerate}

The following theorem outlines several properties of the closed-loop system, which help quantify the aforementioned random times:

\begin{theorem}[\revise{Properties of closed-loop system}]
    Let \( n,m \) be the dimensions of the state and input vectors respectively, and \( P_0, \rho_0, P^*, \rho^* \) be defined in~\eqref{eq:P0}, \eqref{eq:rho0}, \eqref{eq:dare}, \eqref{eq:rho_star} respectively. \revise{Let $\bar{n} = m + n$}, then the following properties hold:
    \begin{enumerate}
        \item For $0 < \delta \leq 1 / 2$, the event
        \begin{align}
            &\mathcal{E}_{\text{noise}}(\delta) := \left\{\max\{\| w_k \|, \| v_k \| \} \leq \vphantom{\sqrt{()}}\right. \nonumber\\
            &\mkern100mu\left. \revise{2\sqrt{\bar{n}\log(k / \delta)}}, \forall k \in \mathbb{N}^*\right\}
            \label{eq:Enoise}
        \end{align}
        occurs with probability at least $1 - 2\delta$.
        \label{item:Enoise}

        \item For $0 < \delta \leq 1/(8n^2)$, the event
        \begin{align}
            & \mathcal{E}_{\text{cov}} (\delta) := \left\{ \left\| \sum_{i=1}^k (w_i w_i^\top - I_n) \right\| \leq \right.\nonumber \\
            &\mkern100mu \left. 7 n\sqrt{k} \log(8n^2k / \delta), \forall k \in \mathbb{N}^* \vphantom{\sum_i^i}\right\}
            \label{eq:Ecov}
        \end{align}
        occurs with probability at least $1 - \delta$.
        \label{item:Ecov}

        \item On the event $\mathcal{E}_{\text{noise}}(\delta)$, it holds
        \begin{equation*}
            \| x_k \| \leq C_x \log(k / \delta), \forall k \in \mathbb{N}^*,
        \end{equation*}
        where
        \begin{equation*}
            C_x = \frac { (\| B \| + 1)(2\sqrt{\revise{\bar{n}}} + 1) \| P_0 \| \| P_0^{-1} \| }{1 - \rho_0^{1/2}}.
        \end{equation*}
        \label{item:xnorm}

        \item For $0 < \delta \leq 1/14$, the event
        \begin{align}
            & \mathcal{E}_{\text{cross}} (\delta) := \left\{ \vphantom{\sum_i^i}\right. \nonumber \\
            & \max \left\{ \left\| \sum_{i=1}^k (w_i + B u_i^{pr})^\top P^*(Ax_i + Bu^{cb}_i) \right\|, \right. \nonumber \\
            &\mkern50mu \left.\left\| \sum_{i=1}^k (u_i^{pr})^\top R u_i^{cb} \right\|, \right. \nonumber \\
            &\mkern50mu \left.\left\| \sum_{i=1}^k w_i^\top P^* B u_i^{pr} \right\|\right\}\leq \nonumber \\
            &\mkern100mu \left. C_{\text{cross}}\sqrt{k} (\log(k/\delta))^2, \forall k \in \mathbb{N}^* \vphantom{\sum_i^i}\right\}
            \label{eq:Ecross}
        \end{align}
        occurs with probability at least $1 - 14\delta$, where
        \begin{align}
            & C_{\text{cross}} = 4\sqrt{\revise{\bar{n}}}  (\| B \| + 1) \cdot \nonumber \\
            & \quad (\| P^* \|\| A \| C_x + \| P^* \|\| B \| + \| R \|).
        \end{align}
        \label{item:Ecross}

        \item
        For $\delta$ satisfying
        \begin{equation}
            0 < \delta < \min \left\{ (38400 C_V)^{-1}, \exp \left(  -22 (m+n)^{1/3} \right) \right\},
            \label{eq:delta_range}
        \end{equation}
        the event
        \begin{align}
            & \mathcal{E}_{\text{est}} (\delta) := \left\{ \left\| \hat{\Theta}_k - \Theta \right\|^2 \leq C_\Theta k^{-1/2}(\log(k/\delta))^3, \right. \nonumber \\
            & \mkern100mu \left. \vphantom{\left\| \hat{\Theta}_k - \Theta \right\|^2} \forall k \geq k_0 \right\},
            \label{eq:Eest}
        \end{align}
        occurs with probability at least $1 - 6\delta$, where
        \begin{align}
            & k_0 = \lceil 4000(m+n)\log(1 / \delta) \rceil, \label{eq:k0} \\
            & C_V = \revise{(C_x + 2\sqrt{\revise{\bar{n}}}+1)^2}, \label{eq:CV} \\
            & C_\Theta = (25600n / 3)(5n/2+2).\label{eq:CTheta}
        \end{align}
        \label{item:Eest}

        \item On the event $ \mathcal{E}_{\text{est}}(\delta)$, it holds
        \begin{equation}
            T_{\text{stab}} \lesssim (\log(1/\delta))^2.
            \label{eq:Tstab_bound}
        \end{equation}
        \label{item:Tstab}

        \item On the event $\mathcal{E}_{\text{noise}}(\delta) \cap \mathcal{E}_{\text{est}}(\delta)$, for any $\alpha > 0$, it holds
        \begin{equation}
            T_{\text{nocb}} \lesssim (1 / \delta)^\alpha,
            \label{eq:Tnocb_bound}
        \end{equation}
        \revise{i.e., as $\delta$ decreases, the growth of $T_{\text{nocb}}$ can be faster than $\operatorname{polylog}(1/\delta)$, but is slower than any polynomial of $1/\delta$.}
        \label{item:Tnocb}
    \end{enumerate}
    \label{thm:prop}
    \revise{
    Here, the notation ``$\lesssim$'' means that the left-hand side is bounded by the right-hand side up to a constant factor, where the constant depends on the system parameters, but not $\delta$ or $\alpha$.
    }
\end{theorem}

\begin{proof}[Proof Sketch]
    \revise{
    We provide a brief and informal sketch of the proof here, and defer the rigorous proof to Appendix~\ref{app:proof_prop}.
    \begin{enumerate}
        \item Corollary of Gaussian concentration.
        \item Corollary of concentration inequalities for quadratic forms of Gaussian vectors, due to Laurent and Massart~\cite{laurent2000adaptive}.
        \item Since the deterministic parts control inputs are bounded by $\log(k)$ under the circuit-breaking mechanism, the random noises are bounded by $\sqrt{\log(k / \delta)}$ according to item \ref{item:Enoise}), and the system is open-loop stable, it follows that the state norm is bounded by $\log(k / \delta)$.
        \item Corollaries of Azuma-Hoeffding inequality~\cite{azuma1967weighted} (noting that all the three summations on the LHS of the inequality are martingales).
        \item We first characterize the least-squares identification error in terms of the maximum and minimum eigenvalues of the regressor covariance matrix $V_k = \sum [x_k^\top\; u_k^\top]^\top[x_k^\top\; u_k^\top]$, using a result due to Abbasi-Yadkori et al.~\cite{abbasi2011online}:
        \begin{equation*}
            \left\| \hat{\Theta}_k - \Theta \right\|^2 \lesssim \frac{\log(\lambda_{\max}(V_k) / \delta)}{\lambda_{\min}(V_k)}.
        \end{equation*}
        Here, $\lambda_{\max}(V_k)$ can grows no faster than $\tilde{\mathcal{O}}(k)$, in light of item~\ref{item:xnorm}) and the log-boundedness of control input imposed by the circuit-breaking mechanism (see Proposition~\ref{prop:max_eig}). Meanwhile, $\lambda_{\min}(V_k)$ grows no slower than $\tilde{\mathcal{O}}(\sqrt{k})$ (see Proposition~\ref{prop:min_eig}), thanks to the sufficient excitation provided to the system by the isotropic probing noise $v_k$. Technically, this lower bound on $\lambda_{\min}(V_k)$ is established using an anti-concentration bound on block martingale small ball~(BMSB) processes~\cite{simchowitz2018learning}, and we show in Lemma~\ref{lemma:verify_bmsb} that $[x_k^\top\; u_k^\top]$ under the proposed controller, despite the presence of the circuit-breaking mechanism, has the BMSB property like in the case of standard certainty equivalent control with probing noise~\cite{simchowitz2020naive,wang2021exact}.
        \item On one hand, when the identification error in item~\ref{item:Eest} is sufficiently small, the certainty equivalent gain $\hat{K}_k$ is sufficiently close to the optimal gain $K^*$, and is therefore stabilizing. On the other hand, when both the modes under switching $A$ and $A+B\hat{K}_k$ are stable, a common Lyapunov function exists when the dwell time $t_k$ is sufficiently large~\cite{ishii2001stabilizing}, which can be reached since $t_k \to\infty$ by our design.
        \item After $T_{\text{stab}}$ is reached, the switched system is stable, and hence $\| x_k \| \lesssim \sqrt{\log(k / \delta)}$ in light of item~\ref{item:xnorm}). Meanwhile, the circuit-breaking threshold $M_k$ grows at the order of $\log(k)$. The quantity $T_{\text{nocb}}$ characterizes the time after which the upper bound on the state norm (and hence the upper bound on $\| u_k^{ce} \|$) drops below the threshold $M_k$, and hence circuit-breaking is no longer triggered. \qedhere
    \end{enumerate}
    }
\end{proof}

\begin{remark}
    \revisetwo{
    Theorem~\ref{thm:prop} offers several key observations:}

    Items \ref{item:Enoise}) and \ref{item:Ecov}) defines two high-probability events on the regularity of noise.
    Item \ref{item:xnorm}) bounds the state norm under regular noise, based on which item \ref{item:Ecross}) bounds the growth of a few cross terms that would be useful in regret analysis.
    Item \ref{item:Eest}) bounds the parameter estimation error, based on which item \ref{item:Tstab}) bounds \( T_{\text{stab}} \).
    Finally, item \ref{item:Tnocb}) states that the circuit-breaking mechanism is triggered only \emph{finitely}, and bounds \( T_{\text{nocb}} \), the time after which the circuit-breaking is not triggered any more.

    \label{rem:prop}
\end{remark}

\begin{remark}
    \revisethree{
    The convergence of system identification results (item \ref{item:Eest}) relies on two factors:
    \begin{enumerate}
        \item The persistence of excitation provided by the probing noise $u_k^{pr}$, which decays at a rate of $\mathcal{O}(k^{-1/4})$.
        \item The bounded magnitude of the feedback control signal $u_k^{cb}$, which is enforced to be $\mathcal{O}(\log k)$ by the circuit-breaking mechanism.
    \end{enumerate}
    These conditions are sufficient to ensure the convergence of parameter estimates, regardless of the stabilizability of individual $(\hat{A}_k, \hat{B}_k)$ pairs at any given time step.
    }
    \label{rem:estimation}
\end{remark}

The following corollary characterizes the tail probabilities of $T_{\text{stab}}$ and $T_{\text{nocb}}$:

\begin{corollary}
    For $T_{\text{stab}}$ defined in~\eqref{eq:Tstab} and $T_{\text{nocb}}$ defined in~\eqref{eq:Tnocb}, as $T\to\infty$, it holds
    \begin{enumerate}
        \item \begin{equation}
            \mathbb{P}(T_{\text{stab}} \geq T) = \mathcal{O}(\exp(-c\sqrt{T})),
            \label{eq:Tstab_tail}
        \end{equation}
        where $c > 0$ is a system-dependent constant.

        \item For any $\alpha > 0$, \begin{equation}
            \mathbb{P}(T_{\text{nocb}} \geq T) = \mathcal{O}(T^{-\alpha}),
            \label{eq:Tnocb_tail}
        \end{equation}
        \revise{
        and consequently,
        \begin{equation}
            T_{\text{nocb}} < \infty, \quad \text{a.s.}
            \label{eq:Tnocb_finite}
        \end{equation}
        }
    \end{enumerate}
    \label{col:time_tail}
\end{corollary}

\begin{proof}
    \revise{
        The bounds~\eqref{eq:Tstab_tail} and~\eqref{eq:Tnocb_tail} follow from invoking items \ref{item:Tstab}) and \ref{item:Tnocb}) of Theorem~\ref{thm:prop} with $\delta$ set to the RHS of the bounds respectively.
        The statement~\eqref{eq:Tnocb_finite} follows from invoking the Borel-Cantelli lemma, noting that the sum of the probabilities of the events $\left\{ T_{\text{nocb}} \geq T \right\}$ is finite under $\alpha > 1$.
    }
\end{proof}

Building upon the properties stated in Theorem~\ref{thm:prop}, a high-probability bound on the regret under the proposed controller can be ensured:

\begin{theorem}[\revise{Non-asymptotic regret bound}]
    Given a probability $\delta$, there exists a constant $T_{0} \lesssim (1 / \delta)^{1/4}$, such that for any fixed $T > T_0$, it holds with probability at least $1 - \delta$ that the regret defined in~\eqref{eq:regret} satisfies
    \begin{equation}
        \mathcal{R}(T) \lesssim (1 / \delta)^{1/4} + \sqrt{T}(\log(T / \delta))^5.
        \label{eq:R_prob}
    \end{equation}
    \label{thm:R_prob}
\end{theorem}

\begin{proof}[Proof Sketch]
    \revise{
    We provide a brief and informal sketch of the proof here, and defer the rigorous proof to Appendix~\ref{app:proof_R_prob}.
    }

    \revise{
    With some algebra, the regret can be decomposed into (see Proposition~\ref{prop:regret_decomposition}):
    \begin{align*}
        & \mathcal{R}(T) \equalsim \sum_{k=1}^T \| K_k - K^* \|^2 \| x_k \|^2 + \\
        & \quad \sum (\text{cross terms in~\eqref{eq:Ecross}}) + \sum (\text{quadratic terms of noise}),
    \end{align*}
    where the quadratic dependence on $\| K_k - K^* \|$ arises from a perturbation result on the Riccati equation (see Lemma~\ref{lemma:perturb_K}).
    Since the latter two terms can be bounded by $\tilde{\mathcal{O}}(\sqrt{T})$ according to Theorem~\ref{thm:prop}, items~\ref{item:Ecov}) and~\ref{item:Ecross}), the main task is to bound the first term.
    }

    \revise{
        We can further split the summation into two parts according to the time $T_{\text{stab}}$:
        \begin{align*}
            & \sum_{k = 1}^T \| K_k - K^* \|^2 \| x_k \|^2 = \\
            & \quad \underbrace{\sum_{k = 1}^{T_{\text{stab}}} \| K_k - K^* \|^2 \| x_k \|^2}_{\text{Term I}}  + \underbrace{\sum_{k = T_{\text{stab}} + 1}^T \| K_k - K^* \|^2 \| x_k \|^2}_{\text{Term II}}.
        \end{align*}
    }

    \revise{
        To bound Term I, note that $\| K_k \|$ is bounded by a constant by design, and that $\| x_k \| \sim \tilde{\mathcal{O}}(1)$ according to Theorem~\ref{thm:prop}, item~\ref{item:xnorm}), Term I can be bounded by $\tilde{\mathcal{O}}(T_{\text{stab}})$. Invoking Theorem~\ref{thm:prop}, item~\ref{item:Tstab}) with $\alpha = 1/4$ leads to the $(1/\delta)^{1/4}$ dependence in the regret bound. The reason for the choice $\alpha = 1/4$ will be evident from the proofs of Corollary~\ref{thm:R_tail} and Theorem~\ref{thm:main}: it ensures that $\delta(T)$ for different $T$'s are summable while keeping $(1/\delta)^\alpha$ at the order of $\sqrt{T}$, contributing to the final almost surely $\tilde{\mathcal{O}}(\sqrt{T})$ regret bound.
    }

    \revise{
        To bound Term II, note that $K_k = \hat{K}_k$ for all $k > T_{\text{nocb}}$ by definition of $T_{\text{nocb}}$, i.e., the certainty equivalent gain is always used. The error of the certainty equivalent gain is bounded by $\tilde{\mathcal{O}}(1/\sqrt{T})$ according to Theorem~\ref{thm:prop}, item~\ref{item:Eest}), summing up to $\tilde{\mathcal{O}}(\sqrt{T})$.
    }
\end{proof}

As corollaries of Theorem~\ref{thm:R_prob}, one can obtain a bound on the tail probability of $\mathcal{R}(T)$ (\revise{Corollary}~\ref{thm:R_tail}), and the main conclusion of this work, i.e., an almost sure bound on $\mathcal{R}(T)$ (Theorem~\ref{thm:main}):

\begin{corollary}
    For sufficiently large $T$, it holds
    \begin{equation*}
        \mathbb{P}\left( \mathcal{R}(T) \geq C_R \sqrt{T}(\log ( T ))^5\right) \leq \frac{1}{T^2},
    \end{equation*}
    where $C_R$ is a system-dependent constant.
    \label{thm:R_tail}
\end{corollary}

\begin{proof}
    The conclusion follows from invoking Theorem~\ref{thm:R_prob} with $\delta = 1 / T^2$.
\end{proof}

\begin{theorem}[\revise{Asymptotic, almost sure regret bound}]
    It holds almost surely that
    \begin{equation*}
        \mathcal{R}(T) = \tilde{\mathcal{O}}\left(\sqrt{T}\right).
    \end{equation*}
    \label{thm:main}
\end{theorem}

\begin{proof}
    By Theorem~\ref{thm:R_tail}, we have
    \begin{equation*}
        \sum_{T = 1}^\infty \mathbb{P}\left( \mathcal{R}(T) \geq C_R \sqrt{T}(\log ( T ))^5\right) < +\infty,
    \end{equation*}
    \revise{i.e., the sum of probabilities of the events $\left\{ \mathcal{R}(T) \geq C_R \sqrt{T}(\log ( T ))^5 \right\}$ is finite.
    By Borel-Cantelli lemma, it holds almost surely that the aforementioned event occurs finitely often,}
    i.e.,
    \begin{equation*}
        \mathcal{R}(T) = \tilde{\mathcal{O}}\left(\sqrt{T}\right), \quad \text{a.s.} \qedhere
    \end{equation*}
\end{proof}

\begin{remark}
    \revisethree{The controllability of estimated system matrices $(\hat{A}_k, \hat{B}_k)$ does not affect the almost sure regret bound. The set of uncontrollable $(A, B)$ pairs forms a proper algebraic variety, which has Lebesgue measure zero~\cite{dion2003generic}. According to~\eqref{eq:ols} and the presence of Gaussian noise, the random matrices $(\hat{A}_k, \hat{B}_k)$ have a distribution that is absolutely continuous with respect to the Lebesgue measure on the space of all matrix pairs. Therefore, for any fixed $k$, the probability of $(\hat{A}_k, \hat{B}_k)$ being uncontrollable is zero. By the countable additivity of probability measures, this extends to all $k$ simultaneously, ensuring that uncontrollability events do not affect our almost sure results.
    Furthermore, as shown in the proof of Proposition~\ref{prop:regret_terms}, on the high-probability event $\mathcal{E}_{\text{est}}$, $(\hat{A}_k, \hat{B}_k)$ converges to $(A, B)$, which implies that $(\hat{A}_k, \hat{B}_k)$ is controllable for sufficiently large $k$. This provides a more formal alternative to the Lebesgue measure argument, demonstrating that the potential uncontrollability of $(\hat{A}_k, \hat{B}_k)$ for some $k$ does not affect our almost sure regret bound.
    }
    \label{rem:controllability}
\end{remark}

\subsection{\revise{Extension to open-loop unstable systems}}
\label{sec:unstable}

\revisetwo{
    This subsection extends the main results to the case where the system is open-loop unstable, but pre-stabilized by a known feedback gain $K_0$}, i.e., $\rho(A) \geq 1$ but $\rho(A+BK_0) < 1$, which is a common assumption in the literature~\cite{simchowitz2020naive,wang2021exact}. Such $K_0$ can be classical controllers commonly used in engineering practice~\cite{sukede2015auto,somefun2021dilemma}, or can result from adaptive or learning-based stabilization~\cite{maartensson1986adaptive,faradonbeh2018finite,hu2022sample}.

\revise{
    In this case, the system can be rewritten as:
    \begin{equation}
        x_{k+1} = A' x_k + B u_k' + w_k,
        \label{eq:K0_sys}
    \end{equation}
    where \( A' = A+BK_0, u_k' = u_k - K_0 x_k \).
    It follows that the optimal control law for the system~\eqref{eq:K0_sys} is \( u_k' = (K^*-K_0) x_k \), where $K^*$ is the optimal control gain~\eqref{eq:K} determined by $(A, B, Q, R)$.
    Therefore, Algorithm~\ref{alg:main} applied to the system~\eqref{eq:K0_sys} will lead to the control law
    \begin{equation*}
        u_k' = \begin{cases}
            (\hat{K}_k - K_0) x_k + u_k^{pr}, & \text{if circuit breaking is not active}, \\
            u_k^{pr}, & \text{otherwise},
        \end{cases}
    \end{equation*}
    where $\hat{K}_k$ is certainty equivalent gain determined by $(\hat{A}, \hat{B}, Q, R)$, $u_k^{pr}$ is the probing noise, and similarly to Lines~\ref{line:switching_start}-\ref{line:switching_end} in Algorithm~\ref{alg:main}, circuit breaking is triggered when $\max\{ \| (\hat{K}_k - K_0) x_k \|, \| \hat{K}_k - K_0 \| \} > M_k$ and lasts for $t_k$ steps.
}

\revise{
    To show that the proposed regret bound still holds, define $X_k = [x_k^\top\; x_{k-1}^\top\; u_{k-1}'^\top]^\top$. It follows that
    \begin{equation*}
        X_{k+1} = \begin{bmatrix}
            A' & 0 & 0 \\
            I & 0 & 0 \\
            0 & 0 & 0
        \end{bmatrix} X_k + \begin{bmatrix}
            B \\
            0 \\
            I
        \end{bmatrix} u_k' + \begin{bmatrix}
            I \\
            0 \\
            0
        \end{bmatrix} w_k,
    \end{equation*}
    and the stepwise cost can be equivalently written as
    \begin{align*}
        &x_k^\top Q x_k + u_k^\top R u_k = \\
        & \qquad X_{k+1}^\top \begin{bmatrix}
            0 & 0 & 0 \\
            0 & Q+K_0^\top R K_0 & K_0^\top R \\
            0 & R K_0 & R
        \end{bmatrix}X_{k+1} + u_k'^\top 0_{m \times m} u_k'.
    \end{align*}
    The problem is then reduced to the case of open-loop stable system with standard LQR cost. Therefore, the proposed almost surely $\mathcal{\tilde{O}}(\sqrt{T})$ regret still holds.
}

\section{Simulation}
\label{sec:simulation}

This section empirically validates the proposed controller and its almost sure regret bound using numerical examples.

\subsection{\revise{Result on an open-loop stable system}}
\label{sec:tep}

\revise{
In this subsection, the proposed controller is tested on the Tennessee Eastman Process (TEP)~\cite{downs1993plant}.}
In particular, we consider a simplified version of TEP similar to the one in~\cite{liu2020online}, with full state feedback.
The system is open-loop stable \revisetwo{with a state dimension  of \( n = 8 \) and an input dimension of \( m = 4 \). The process noise follows \( w_k \stackrel{\text { i.i.d. }}{\sim} \mathcal{N}(0, 16I_n) \).}
The weight matrices of LQR are chosen to be \( Q = I_n \) and \( R = I_m \).
The plant under the proposed controller is simulated for 10000 independent trials, each with \revise{$10^9$} steps.
As mentioned in Remark~\ref{rem:log_often}, the certainty equivalent gain \( \hat{K}_k \) is updated only at steps \( k = 2^i, i \in \mathbb{N}^* \) for the sake of fast computation.

The evolution of regret against time is plotted in Fig.~\ref{fig:avg_regret}. \revisetwo{To facilitate observation}, we plot the \emph{relative average regret} \( \mathcal{R}(T) / (TJ^*)\) against the total time step \( T \), where \( J^* \) is the optimal cost. Fig.~\ref{fig:avg_regret} shows 5 among the 10000 trials, from which \revisetwo{a \( 1 / \sqrt{T} \) convergence rate of the relative average regret is observed (i.e., a 1 order-of-magnitude increase in \( T \) corresponds to a 0.5 order-of-magnitude decrease in \( \mathcal{R}(T) / (TJ^*) \)), matching the theoretical \( \sqrt{T} \) growth rate of regret.}
\revise{To examine statistical properties, we sorted the trials by the average regret at the last step, and plotted the worst, median and mean cases in Fig.~\ref{fig:minmax}.} One can observe that the average regret converge to zero even in the worst case, which validates the almost-sure guarantee in Theorem~\ref{thm:main}. \revise{The $\sqrt{T}$-rate is also empirically validated by Fig.~\ref{fig:minmax_normalized}, which shows that $\mathcal{R}(T)/\sqrt{T}$ does not explode over time.}

\begin{figure}[!htbp]
    \centering
    \begin{subfigure}{\columnwidth}
        \centering
  \tikzsetnextfilename{figures/fig_samples.tex}%
  \input{figures/fig_samples.tex}%

        \caption{Five random sample paths}
        \label{fig:sample}
    \end{subfigure}
    \begin{subfigure}{\columnwidth}
        \centering
  \tikzsetnextfilename{figures/fig_minmax.tex}%
  \input{figures/fig_minmax.tex}%

        \caption{Worst, median and best cases among all sample paths \revise{(sorted according to the average regret at the last time step)}}
        \label{fig:minmax}
    \end{subfigure}
    \caption{Double-log plot of average regret against time step}
    \label{fig:avg_regret}
\end{figure}

\begin{figure}[!htbp]
    \centering
  \tikzsetnextfilename{figures/fig_minmax_normalized.tex}%
  \input{figures/fig_minmax_normalized.tex}%

    \vspace{-0.5cm}
    \caption{\revise{Regret divided by $\sqrt{T}$ against time step}}
    \label{fig:minmax_normalized}
\end{figure}

\revisetwo{
A key insight is that the circuit-breaking mechanism is triggered only finitely, with the time of the last trigger \( T_{\text{nocb}} \) exhibiting a super-polynomial tail, as stated with Corollary~\ref{col:time_tail}.
This is consistent with the observations}: among all the 10000 trials, circuit-breaking is never triggered after step \( 1.4\times 10^6 \), and a histogram of \( T_{\text{nocb}} \) is shown in Fig.~\ref{fig:Tnocb}, from which one can observe that the empirical distribution of \( T_{\text{nocb}} \) has a fast decaying tail.

\begin{figure}[!htbp]
    \centering
  \tikzsetnextfilename{figures/fig_Tnocb.tex}%
  \input{figures/fig_Tnocb.tex}%

    \caption{Histogram of \( T_{\text{nocb}} \) among all sample paths}
    \label{fig:Tnocb}
\end{figure}

\subsection{\revise{Result on an open-loop unstable system}}
\label{sec:unstable_example}

\revisetwo{
    We also tested the cart-pole inverted pendulum~\cite{geva1993cartpole} system, an open-loop unstable example. The system was linearized around the origin and discretized with time step $0.1\mathrm{s}$, resulting in a state dimension of $n=4$ and input dimension of $m=1$. The process noise follows $w_k \stackrel{\text { i.i.d. }}{\sim} \mathcal{N}(0, 0.02I_n)$. The system is pre-stabilized with a manually crafted gain $K_0 = \begin{bmatrix} 1 & 1 & 20 & 2 \end{bmatrix}$. Other experiments settings are the same as the TEP example in subsection~\ref{sec:tep}.
}

\revise{
    The relative average regret is shown in Fig.~\ref{fig:unstable_sample}. A similar convergence pattern as in the TEP example (Fig.~\ref{fig:sample}) is observed, which validates the claim in Section~\ref{sec:unstable} that the proposed results also hold for pre-stabilized open-loop unstable systems. Similar phenomenon as in Fig.~\ref{fig:minmax}, Fig.~\ref{fig:minmax_normalized} and Fig.~\ref{fig:Tnocb} are also observed in the cart-pole example, which are omitted for brevity.
}

\begin{figure}[!htbp]
    \centering
  \tikzsetnextfilename{figures/fig_samples_unstable.tex}%
  \input{figures/fig_samples_unstable.tex}%

    \caption{\revise{Double-log plot of average regret against time step in open-loop unstable cart-pole system. Five random sample paths are shown.}}
    \label{fig:unstable_sample}
\end{figure}

\section{Conclusion}
\label{sec:conclusion}

\revisetwo{
This paper proposes an adaptive LQR controller capable of achieving \( \tilde{\mathcal{O}}(\sqrt{T}) \) regret almost surely.
A key aspect of the controller design is a circuit-breaking mechanism, which ensures the convergence of the parameter estimate while being triggered only finitely often, thus having negligible effect on asymptotic performance.
}

\revisetwo{
Future research directions include: i) extending the circuit-breaking mechanism to the partially observed LQG setting, and ii) considering more general process noise distributions, such as sub-Weibull distributions~\cite{vladimirova2020sub}. This analysis could build on identification bounds in~\cite{faradonbeh2020input} and switching costs in~\cite{lu2022safe}.
}

\bibliographystyle{IEEEtran}
\bibliography{ref.bib}

\appendices

\section{Proof of Theorem~\ref{thm:prop}}
\label{app:proof_prop}

\subsection{Proof of Theorem~\ref{thm:prop}, item~\ref{item:Enoise})}

\begin{proof}
    Since $w_k \sim \mathcal{N}(0, I_n)$, it holds $\left\| w_k \right\|^2 \sim \chi^2(n)$. Applying the Chernoff bound, for any $a > 0$, and any $0 < t < 1/2$, it holds
    \begin{equation*}
        \mathbb{P}(X \geq a) \leq \mathbb{E}\left[\mathrm{e}^{t X} / \mathrm{e}^{t a}\right]=(1-2 t)^{-n / 2} \exp (-t a).
    \end{equation*}
    Choosing $t = 1/4$, we have
    \begin{equation}
        \mathbb{P}(\|w_k\| \geq a) \leq 2^{n/2} \exp(-a^2 / 4)
        \label{eq:w_tail}
    \end{equation}
    for any $k \in \mathbb{N}^*$ and $a > 0$.
    Invoking~\eqref{eq:w_tail} with $a_k = 2(\log(ck^2 / \delta))^{1/2}$, where $c = 2^{n/2}\pi^2 / 6$, we have
    \begin{equation*}
        \sum_{k=1}^\infty \mathbb{P}(\| w_k \| \geq a_k) \leq \sum_{k=1}^\infty 2^{n/2}c^{-1}\delta / k^2 = \delta,
    \end{equation*}
    i.e., it holds with probability at least $1 - \delta$ that
    \begin{equation}
        \| w_k \| \leq 2(\log(ck^2/\delta))^{1/2} \leq 2\sqrt{n+1}\sqrt{\log(k / \delta)}, \forall k \in \mathbb{N}^*.
        \label{eq:w_uni}
    \end{equation}
    Similarly, it also holds with probability at least $1 - \delta$ that
    \begin{equation}
        \| v_k \|  \leq 2\sqrt{\revise{m}+1}\sqrt{\log(k / \delta)}, \forall k \in \mathbb{N}^*.
        \label{eq:v_uni}
    \end{equation}
    Combining~\eqref{eq:w_uni} and~\eqref{eq:v_uni} leads to the conclusion.
\end{proof}

\subsection{Proof of Theorem~\ref{thm:prop}, item~\ref{item:Ecov})}

We start with a concentration bound on the sum of product of Gaussian random variables:

\begin{lemma}
    Let $X_i \stackrel{\text { i.i.d. }}{\sim} \mathcal{N} (0, 1)$, $Y_i \stackrel{\text { i.i.d. }}{\sim} \mathcal{N} (0, 1)$, and $\{ X_i \}, \{ Y_i \}$ be mutually independent, then
    \begin{enumerate}
        \item With probability at least $1 - 4\delta$,
        \begin{equation}
            \left| \sum_{i=1}^k X_i^2 - k \right| \leq 7\sqrt{k}\log(k / \delta), \forall k \in \mathbb{N}^*.
            \label{eq:sum_x2}
        \end{equation}
        \label{item:sum_x2}

        \item With probability at least $1 - 8\delta$,
        \begin{equation}
            \left| \sum_{i=1}^k X_iY_i \right| \leq 5\sqrt{k}\log(k / \delta), \forall k \in \mathbb{N}^*.
            \label{eq:sum_xy}
        \end{equation}
        \label{item:sum_xy}
    \end{enumerate}
    \label{lemma:subexp}
\end{lemma}

\begin{proof}
    Since $\sum_{i=1}^k X_i^2 \sim \chi^2(k)$, according to~\cite[Lemma 1]{laurent2000adaptive}, for any $a > 0$ and any $k\in \mathbb{N}^*$, it holds
    \begin{equation*}
        \mathbb{P}\left( \left| \sum_{i=1}^k X_i^2 - k \right| \geq 2\sqrt{\revise{k}a} + 2a \right) \leq 2\exp(-a).
    \end{equation*}
    Fix $k$ and choose $a = \log(k^2 / \delta)$, and it follows that with probability at least $1 - 2\delta / k^2$,
    \begin{align}
        \left| \sum_{i=1}^k X_i^2 - k \right| & \leq 2\sqrt{k\log(k^2 / \delta)} + 2 \log(k^2 / \delta) \nonumber \\
        & \leq 7\sqrt{k}\log(k / \delta).
    \end{align}
    Taking the union bound over $k\in \mathbb{N}^*$, one can show that~\eqref{eq:sum_x2} holds with probability at least $1 - 2\delta \sum_{k=1}^\infty (1/k^2) > 1 - 4\delta$, and hence claim~\ref{item:sum_x2}) is proved.

    Since $X_i Y_i = (X_i + Y_i)^2 / 4 + (X_i - Y_i)^2 / 4$, and $X_i + Y_i \stackrel{\text { i.i.d. }}{\sim} \mathcal{N} (0, 2)$, $X_i - Y_i \stackrel{\text { i.i.d. }}{\sim} \mathcal{N} (0, 2)$, claim~\ref{item:sum_xy}) follows from applying claim~\ref{item:sum_x2}) to $\{ (X_i + Y_i) / \sqrt{2} \}$ and $\{ (X_i - Y_i) / \sqrt{2} \}$ respectively and taking the union bound.
\end{proof}

Theorem~\ref{thm:prop}, item~\ref{item:Ecov}) follows from the above lemma:

\begin{proof}
    Applying Lemma~\ref{lemma:subexp}, item~\ref{item:sum_x2}) to the diagonal elements, and Lemma~\ref{lemma:subexp}, item~\ref{item:sum_xy}) to the off-diagonal elements of $\sum_{i=1}^k (w_i w_i^\top - I)$, and taking the union bound, one can show that with probability at least $1 - 8n^2\delta$,
    \begin{equation*}
        \sum_{i=1}^k (w_i w_i^\top - I_n) \leq 7\sqrt{k} \log(k / \delta),
    \end{equation*}
    where the inequality holds component-wise. Hence, with probability at least $1 - 8n^2\delta$,
    \begin{align}
        \left\| \sum_{i=1}^k (w_i w_i^\top - I_n) \right\|^2 & \leq \left\| \sum_{i=1}^k (w_i w_i^\top - I_n) \right\|_F^2 \nonumber \\
        & \leq n^2(7\sqrt{k} \log(k / \delta)),
    \end{align}
    and scaling the failure probability leads to the conclusion.
\end{proof}

\subsection{Proof of Theorem~\ref{thm:prop}, item~\ref{item:xnorm})}

\begin{proof}
    Notice
    \begin{equation}
        x_k = A^{k-2} d_1 + A^{k-3} d_2 + \cdots + d_{k-1},
        \label{eq:state_decomp}
    \end{equation}
    where $d_k = B(u^{cb}_k + u^{pr}_k) + w_k$. On $\mathcal{E}_{\text{noise}}(\delta)$, it holds
    \begin{align}
        \| d_k \| & \leq \| B \| \log(k) + 2(\| B \| + 1)\revise{\sqrt{\bar{n}\log(k / \delta)}} \nonumber \\
        & \leq (\| B \| + 1)\revise{(2\sqrt{\bar{n}}+1)}\log(k / \delta).
    \end{align}
    Furthermore, from~\eqref{eq:P0} and~\eqref{eq:state_decomp}, it holds
    \begin{align}
        \| x_k \|_{\revise{P_0}} & \leq \rho_0^{(k-2)/2}\| d_1 \|_{\revise{P_0}} + \rho_0^{(k-3)/2}\| d_2 \|_{\revise{P_0}} + \cdots + \| d_{k-1} \|_{\revise{P_0}} \nonumber \\
        & \leq \frac{1}{1 - \rho_0^{1/2}} \|d_k \|_{\revise{P_0}},
        \label{eq:xk_bound}
    \end{align}
    from which the conclusion follows.
\end{proof}

\subsection{Proof of Theorem~\ref{thm:prop}, item~\ref{item:Ecross})}

This result is a corollary of a time-uniform version of Azuma-Hoeffding inequality~\cite{azuma1967weighted}, stated below:

\begin{lemma}
    Let $\{ \phi_k \}_{k\geq 1}$ be a martingale difference sequence adapted to the filtration $\{ \mathcal{F}_k \}$ satisfying $| \phi_k | \leq d_k$ a.s., then with probability at least $1 - 4\delta$, it holds
    \begin{equation}
        \left| \sum_{i=1}^k  \phi_i \right| \leq 2\sqrt{\sum_{i=1}^k d_i^2 \log(k / \delta)}, \forall k.
        \label{eq:azuma}
    \end{equation}
    \label{lemma:azuma}
\end{lemma}

\begin{proof}
    By Azuma-Hoeffding inequality~\cite{azuma1967weighted}, for a fixed $k$, it holds with probability at least $1 - 2\delta / k^2$ that
    \begin{equation}
        \left| \sum_{i=1}^k  \phi_i \right| \leq \sqrt{2\sum_{i=1}^k d_i^2 \log(k^2 / \delta)} \leq 2\sqrt{\sum_{i=1}^k d_i^2 \log(k / \delta)}.
        \label{eq:azuma_1}
    \end{equation}
    Taking the union bound over $k\in \mathbb{N}^*$, one can prove that~\eqref{eq:azuma} holds with probability at least $1 - 2\delta \sum_{k=1}^\infty (1/k^2) > 1 - 4\delta$.
\end{proof}

Theorem~\ref{thm:prop}, item~\ref{item:Ecross}) follows from the above lemma:

\begin{proof}
    According to Theorem~\ref{thm:prop}, item~\ref{item:Enoise}), we only need to prove
    \begin{equation}
        \mathbb{P}( \mathcal{E}_{\text{cross}}(\delta) \mid \mathcal{E}_{\text{noise}}(\delta)) \geq 1 - 12\delta,
        \label{eq:prob_cross}
    \end{equation}
    \revise{since when~\eqref{eq:prob_cross} holds, we have $\mathbb{P}(\bar{\mathcal{E}}_{\text{cross}}(\delta))= \mathbb{P}(\bar{\mathcal{E}}_{\text{cross}}(\delta) \mid \mathcal{E}_{\text{noise}}(\delta)) \mathbb{P}(\mathcal{E}_{\text{noise}}(\delta)) + \mathbb{P}(\bar{\mathcal{E}}_{\text{cross}}(\delta) \mid \bar{\mathcal{E}}_{\text{noise}}(\delta))\mathbb{P}(\bar{\mathcal{E}}_{\text{noise}}(\delta)) \leq \mathbb{P}(\bar{\mathcal{E}}_{\text{cross}}(\delta) \mid \mathcal{E}_{\text{noise}}(\delta)) + \mathbb{P}(\bar{\mathcal{E}}_{\text{noise}}(\delta)) \leq 14\delta$.}
    Therefore, we condition the remainder of the proof upon the event $\mathcal{E}_{\text{noise}}(\delta)$.

    Let $\mathcal{F}_k$ be the $\sigma$-algebra generated by $v_1,\allowbreak w_1,\allowbreak v_2, w_2, \ldots,\allowbreak v_{k-1},\allowbreak w_{k-1}, v_k$.
    Since $x_k \in \mathcal{F}_{k-1}$, $u_k^{cb} \in \mathcal{F}_{k - 1}$, and $\mathbb{E}[w_k \mid \mathcal{F}_{k-1}] = 0, \mathbb{E}[v_k \mid \mathcal{F}_{k-1}]=0$ due to symmetry, it holds
    \begin{align}
        & \mathbb{E}\left[(w_k+B u_k^{pr})^\top P^* \left(A x_k + B u_k^{cb}\right) \big| \mathcal{F}_k\right] \nonumber \\
        = & \mathbb{E}\left[w_k + k^{-1/4}Bv_k \big| \mathcal{F}_{k-1}\right]^\top P^* \left(A x_k + B u_k^{cb}\right) = 0,
    \end{align}
    i.e., $\{ (w_k + Bu_k^{pr})^\top P^* (A x_k + B u_k^{cb}) \}$ is a martingale difference sequence adapted to the filtration $\{ \mathcal{F}_k \}$. Furthermore, by Theorem~\ref{thm:prop}, item~\ref{item:xnorm}), we have
    \begin{align}
        & \left\| (w_k + Bu_k^{pr})^\top P^* \left(A x_k + B u_k^{cb}\right) \right\| \nonumber \\
        \leq & (\| w_k \| + \| B \| \| v_k \|) \| P^* \| \left(\| A \| \| x_k \| + \| B \| \left\| u_k^{cb} \right\|\right) \nonumber \\
        \leq & \frac{1}{2} C_{\text{cross}} \left(\log(k / \delta)\right)^{3/2}.
    \end{align}
    Therefore, by applying Lemma~\ref{lemma:azuma}, it holds
    \begin{align}
        & \mathbb{P} \left(\left\| \sum_{i=1}^k (w_i + B u_i^{pr})^\top P^*(Ax_i + Bu^{cb}_i) \right\| \geq \right. \nonumber \\
        & \quad \left. \vphantom{\sum_i^i} C_{\text{cross}}\sqrt{k} (\log(k/\delta))^2, \forall k \in \mathbb{N}^* \right| \left. \mathcal{E}_{\text{noise}}(\delta) \vphantom{\sum_i^i} \right) \leq 4\delta.
        \label{eq:prob_cross_1}
    \end{align}
    Following a similar argument, it also holds
    \begin{align}
        & \mathbb{P} \left(\left\| \sum_{i=1}^k (u_i^{pr})^\top R u_i^{cb} \right\| \geq \right. \nonumber \\
        & \quad \left. \vphantom{\sum_i^i} C_{\text{cross}}\sqrt{k} (\log(k/\delta))^2, \forall k \in \mathbb{N}^* \right| \left. \mathcal{E}_{\text{noise}}(\delta) \vphantom{\sum_i^i} \right) \leq 4\delta,
        \label{eq:prob_cross_2}
    \end{align}
    and
    \begin{align}
        & \mathbb{P} \left(\left\| \sum_{i=1}^k w_i^\top P^* B u_i^{pr} \right\| \geq \right. \nonumber \\
        & \quad \left. \vphantom{\sum_i^i} C_{\text{cross}}\sqrt{k} (\log(k/\delta))^2, \forall k \in \mathbb{N}^* \right| \left. \mathcal{E}_{\text{noise}}(\delta) \vphantom{\sum_i^i} \right) \leq 4\delta.
        \label{eq:prob_cross_3}
    \end{align}
    Hence, the conclusion follows from combining~\eqref{eq:prob_cross_1}, \eqref{eq:prob_cross_2} and~\eqref{eq:prob_cross_3}.
\end{proof}

\subsection{Proof of Theorem~\ref{thm:prop}, item~\ref{item:Eest})}
\label{sec:Eest_proof}

This subsection is devoted to deriving the time-uniform upper bound on estimation error stated in Theorem~\ref{thm:prop}, item~\ref{item:Eest}).
Throughout this subsection, we denote $\Theta = [A\quad B]$, $\hat{\Theta}_k = [\hat{A}_k \quad \hat{B}_k]$, $z_k = [x_k^\top \quad u_k^\top]^\top$, and $V_k = \sum_{i=1}^{k - 2} z_i z_i^\top$.

The proof can be split into three parts: firstly, we characterize the estimation error $\| \hat{\Theta}_k - \Theta \|$ in terms of the maximum and minimum eigenvalues of the regressor covariance matrix $V_k$, using a result in martingale least squares~\cite{abbasi2011online}.
Secondly, an upper bound on the $\|V_k\|$, which is a consequence of the non-explosiveness of states, can be derived as \revise{a} corollary of Theorem~\ref{thm:prop}, item~\ref{item:xnorm}).
Finally, an upper bound on $\| V_k^{-1} \|$, or equivalently a lower bound on the minimum eigenvalue of $V_k$, which is a consequence of sufficient excitation of the system, can be proved using an anti-concentration bound on block martingale small-ball~(BMSB) processes~\cite{simchowitz2018learning}.
The three parts would be discussed respectively below.

\subsubsection{Upper bound on least squares error, in terms of $V_k$}

\newcommand{\abbasi}{\cite[Corollary~1 of Theorem~3]{abbasi2011online}}
\begin{lemma}[\abbasi]
    Let
    \begin{equation*}
        S_k = \sum_{i=1}^k \eta_i m_{i-1}, U_k = \sum_{i=1}^k m_{i-1}m_{i-1}^\top,
    \end{equation*}
    where $\{\mathcal{F}_k\}_{k\in  \mathbb{N}^*}$ is a filtration, $\{\eta_k\}_{k\in  \mathbb{N}^*}$ is a random scalar sequence with $\eta_k \mid \mathcal{F}_k$ being conditionally $\sigma^2$-sub-Gaussian, and $\{m_k\}_{k\in  \mathbb{N}^*}$ is a random vector sequence with $m_k \in \mathcal{F}_k$. Then with probability at least $1 - \delta$,
    \begin{align}
        & \|S_k\|_{(U_0+U_k)^{-1}} \leq 2\sigma^2 \cdot \nonumber \\
        & \quad \log\left(\det(U_0)^{-1/2} \det(U_0+U_k)^{1/2} / \delta\right), \quad \forall k \in \mathbb{N}^*,
    \end{align}
    where $U_0 \succ 0$ is an arbitrarily chosen constant positive semi-definite matrix.
    \label{lemma:abbasi}
\end{lemma}

\begin{proposition}
    With probability at least $1 - \delta$,
    \begin{align}
        & \left\| \hat{\Theta}_k - \Theta \right\|^2 \leq 2n \left\| V_k^{-1} \right\|\left[ \log \left( \frac{n}{\delta} \right) + \frac{n}{2} \log(1 + \| V_k \|) \right], \nonumber \\
        & \forall k \geq m + n + 2.
        \label{eq:Theta_bound}
    \end{align}
    \label{prop:Els}
\end{proposition}

\begin{proof}
    Let $\mathcal{F}_k$ be the $\sigma$-algebra generated by $v_1,\allowbreak w_1,\allowbreak v_2,\allowbreak w_2,\allowbreak \ldots,\allowbreak v_{k-1},\allowbreak w_{k-1}, v_k$.
    From $x_{k+1} = \Theta z_k + w_k$, we have
    \begin{equation*}
        \hat{\Theta}_k - \Theta = \left(\sum_{i=1}^{k-2} w_i z_{i}^\top\right)\left(\sum_{i=1}^{k-2} z_{i} z_{i}^\top\right)^\dag,
    \end{equation*}
    where $w_k | \mathcal{F}_{k} \sim \mathcal{N}(0, I_n), z_k \in \mathcal{F}_k$. With $V_k = \sum_{i=1}^{k-2} z_{i} z_{i}^\top$, we have $V_{\revise{k}}\succ 0$ a.s. for $k \geq m+n+2$.
    Now we can apply Lemma~\ref{lemma:abbasi} to each row of $\hat{\Theta}_k - \Theta$: for each of $j = 1,\ldots, n$, let $S_{j,k} = \sum_{i=1}^{k-2} (e_j^\top w_i) z_{i}$, where $e_j$ is the $j$-th standard unit vector.
    By invoking Lemma~\ref{lemma:abbasi} with $U_k = V_k$ and $U_0 = I_{m+n}$, we have: with probability at least $1 - \delta$,
    \begin{align}
        & \left\|e_j^\top(\hat{\Theta}_k - \Theta) \right\|^2 = \left\| V_k^{-1} S_{j,k} \right\|
        \leq \left\| V_k ^{-1} \right\| \| S_{j,k} \|_{(I+V_k)^{-1}} \nonumber \\
        & \leq 2\left\| V_k^{-1} \right\| \log\left( \det(I+V_k)^{1/2} / \delta\right) \nonumber \\
        & \leq 2 \left\| V_k^{-1} \right\| \left[ \log\left(\frac{1}{\delta}\right) + \frac{n}{2}\log(1 + \| V_k \|) \right].
    \end{align}
    Taking the union bound over $j = 1,\ldots,n$, we have: with probability at least $1 - n\delta$,
    \begin{align}
        & \left\| \hat{\Theta}_k - \Theta \right\|^2 \leq \left\| \hat{\Theta}_k - \Theta \right\|_F^2 \nonumber \\
        \leq & \sum_{j=1}^n \left\| e_j^\top (\hat{\Theta}_k - \Theta) \right\|^2 \nonumber \\
        \leq & 2n \left\| V_k^{-1} \right\| \left[ \log\left(\frac{1}{\delta}\right) + \frac{n}{2}\log(1 + \| V_k \|) \right].
        \label{eq:Theta_bound_1}
    \end{align}
    Scaling the failure probability results in the conclusion.
\end{proof}

\subsubsection{Upper bound on $\| V_k \|$}

\begin{proposition}
    On the event $\mathcal{E}_{\text{noise}}(\delta)$ defined in~\eqref{eq:Enoise}, it holds
    \begin{equation*}
        \| V_k \| \leq C_V k (\log(k / \delta))^2,
    \end{equation*}
    where $C_V$ is defined in~\eqref{eq:CV}.
    \label{prop:max_eig}
\end{proposition}

\begin{proof}
    On $\mathcal{E}_{\text{noise}}(\delta)$, we have
    \begin{equation*}
        \| u_k \| \leq \log(k) + 2\revise{\sqrt{\bar{n}\log(k / \delta)}},
    \end{equation*}
    and by Theorem~\ref{thm:prop}, item~\ref{item:xnorm}), we have
    \begin{equation*}
        \| x_k \| \leq C_x\log(k / \delta).
    \end{equation*}
    Hence,
    \begin{equation*}
        \| z_k \| \leq \| x_k \| + \| u_k \| \leq \sqrt{C_V}\log(k / \delta),
    \end{equation*}
    which implies
    \begin{equation*}
        \| V_k \| \leq \sum_{i = 1}^{k-2}\| z_k \|^2 \leq C_V k(\log(k / \delta))^2.
    \end{equation*}
\end{proof}

\subsubsection{Upper bound on $\| V_k^{-1} \|$}

We shall borrow the techniques of analyzing BMSB processes from Simchowitz et al.~\cite{simchowitz2018learning} to bound $\| V_k^{-1} \|$. The BMSB process is defined as follows:

\newcommand{\bmsb}{\cite[Definition~2.1]{simchowitz2018learning}}
\begin{definition}[\bmsb]
    \revise{
    Suppose that $\left\{\phi_{i}\right\}_{i \in \mathbb{N}^*}$ is a real-valued stochastic process adapted to the filtration $\left\{\mathcal{F}_{i}\right\}$. We say the process $\left\{\phi_{i}\right\}$ satisfies the $(l, \nu, p)$ block martingale small-ball (BMSB) condition if:
    \begin{equation*}
        \frac{1}{l} \sum_{j=1}^{l} \mathbb{P}\left(\left|\phi_{i+j}\right| \geq \nu \mid \mathcal{F}_{i}\right) \geq p, \forall i \in \mathbb{N}^*.
    \end{equation*}
    }
    \label{def:bmsb}
\end{definition}

The following lemma verifies that $\{ z_k \}$, projected along an arbitrary direction, is BMSB:

\begin{lemma}
    For any $\mu \in \mathbb{S}^{m+n}$ and any \( k\geq 3 \), the process $\{\langle z_i, \mu \rangle\}_{i=1}^{k-2}$ satisfies the $\left(1, \frac{k^{-1/4}}{2(\log(k)+1)} , 3/10\right)$ BMSB condition.
    \label{lemma:verify_bmsb}
\end{lemma}

\begin{proof}
    Let \( \mathcal{F}_i \) be the $\sigma$-algebra generated by \( \{ v_1, w_1, v_2, w_2, \ldots, v_{i-1}, w_{i - 1}, v_i \}\,(i=1,\ldots,k-2) \), then we only need to verify
    \begin{equation}
        \mathbb{P} \left( \left| \langle z_{i+1}, \mu \rangle \right| \geq \frac{k^{-1/4}}{2(\log(k)+1)} \middle| \mathcal{F}_i \right) \geq \frac{3}{10}.
        \label{eq:310}
    \end{equation}

    According to the above definition of \( \mathcal{F}_i \), we have \( x_t \in \mathcal{F}_i \) and \( u_t \in \mathcal{F}_i \) for any \( t = 0, 1, \ldots, i \). Furthermore, we have \( \hat{K}_{i+1} \in \mathcal{F}_i \) since \( \hat{K}_{i+1} \) is a deterministic function of \( \hat{\Theta}_{i+1} \), and the latter only depends on \( \{ x_t \}_{t=1}^{i}, \{ u_t \}_{t=1}^{i}\), according to~\eqref{eq:ols}.

    \revise{
    Let
    \begin{equation}
        K_i = \begin{cases}
            \hat{K}_i & u_i^{cb} = u_i^{ce} \\
            0 & u_i^{cb} = 0, \\
        \end{cases}
    \end{equation}
    such that \( u_i^{cb} = K_i x_i \), then we can split \( z_{i+1} \) into two parts as follows:
    \begin{align}
        z_{i+1} &= \begin{bmatrix}
            x_{i+1} \\
            u_{i+1}
        \end{bmatrix} = \begin{bmatrix}
            x_{i+1} \\
            K_{i+1}x_{i+1} + u_{i+1}^{pr}
        \end{bmatrix} \nonumber\\
        &= \begin{bmatrix}
            I_n \\
            K_{i+1}
        \end{bmatrix}
        (A_i x_i + B_i u_i^{pr}) + \nonumber\\
        & \qquad \begin{bmatrix}
            I_n & 0 \\ K_{i+1} & (i+1)^{-1/4} I_m
        \end{bmatrix} \begin{bmatrix}
            w_i \\ v_{i+1}
        \end{bmatrix},
        \label{eq:z_Gauss}
    \end{align}
    where \( K_{i+1} \in \{ \hat{K}_{i + 1}, 0 \} \). Next discuss the two possibilities of \( K_{i+1} \) respectively:
    }

    \revise{
    \textbf{Case 1}: \( K_{i+1} = \hat{K}_{i + 1} \), then by design it holds \( \| K_{i+1} \| \leq \log(i + 1) \leq \log(k) \). Also note that conditional on \( \mathcal{F}_i \) and \( K_{i+1} = \hat{K}_{i + 1} \), the first term in the RHS of~\eqref{eq:z_Gauss} is deterministic, while the second term is Gaussian with zero mean and whose variance is:
    \begin{align}
        &\mathbb{E} \left[ \left( \langle z_{i+1}, \mu \rangle - \mathbb{E} \left[ \langle z_{i+1}, \mu \rangle \middle| \mathcal{F}_i,K_{i+1} = \hat{K}_{i + 1} \right] \right)^2 \right| \nonumber\\
        &\mkern50mu \left. \vphantom{\left( \langle z_{i+1}, \mu \rangle - \mathbb{E} \left[ \langle z_{i+1}, \mu \rangle \middle| \mathcal{F}_i,K_{i+1} = \hat{K}_{i + 1} \right] \right)^2} \mathcal{F}_i,K_{i+1} = \hat{K}_{i + 1} \right] \nonumber\\
        &= \mathbb{E}  \left[ \mu^\top\begin{bmatrix}
            I_n & 0 \\ K_{i+1} & (i+1)^{-1/4} I_m
        \end{bmatrix} \begin{bmatrix}
            w_i \\ v_{i+1}
        \end{bmatrix} \begin{bmatrix}
            w_i \\ v_{i+1}
        \end{bmatrix}^\top \right.  \nonumber\\
        &\mkern50mu \left. \begin{bmatrix}
            I_n & 0 \\ K_{i+1} & (i+1)^{-1/4} I_m
        \end{bmatrix}^\top \mu \middle| \mathcal{F}_i,K_{i+1} = \hat{K}_{i + 1}\right]  \nonumber\\
        &\geq (i+1)^{-1/2} \mu^\top \begin{bmatrix}
            I_n & K_{i+1}^\top \\
            K_{i+1} & I_m + K_{i+1}K_{i+1}^\top
        \end{bmatrix} \mu.
        \label{eq:var_zmu}
    \end{align}
    To derive a lower bound on the variance of \( \langle z_{i+1}, \mu \rangle \) conditional on \( \mathcal{F}_i \), we bound the minimum singular value of the matrix in the RHS of~\eqref{eq:var_zmu} as follows:
    \begin{align}
        &\sigma_{\min} \left( \begin{bmatrix}
            I_n & K_{i+1}^\top \\
            K_{i+1} & I_m + K_{i+1}K_{i+1}^\top
        \end{bmatrix} \right) \nonumber\\
        &= \left\| \begin{bmatrix}
            I_n & K_{i+1}^\top \\
            K_{i+1} & I_m + K_{i+1}K_{i+1}^\top
        \end{bmatrix}^{-1}  \right\|^{-1} \nonumber\\
        &= \left\| \begin{bmatrix}
            I_n  & -K_{i+1}^\top \\
            0 & I_m
        \end{bmatrix} \begin{bmatrix}
            I_n  & 0 \\
            -K_{i+1} & I_m
        \end{bmatrix} \right\|^{-1} \nonumber\\
        &\geq \frac{1}{2 + \| K_{i+1} \|^2} \geq \frac{1}{4(\log(k) + 1)^2},
    \end{align}
    where the second last inequality follows from the fact that \[  \|M\|^2 \leq \|A\|^2 + \|B\|^2 + \|D\|^2 \text{ for } M = \begin{bmatrix}
        A & B \\ 0 & D
    \end{bmatrix}.\]
    Hence, the RHS of~\eqref{eq:var_zmu} is greater than or equal to \[ \frac{k^{-1/2}}{4(\log(k) + 1)^2}.\]
    The following inequality then follows from the fact that for any $X \sim \mathcal{N}(\mu, \sigma^2)$, it holds $\mathbb{P}(|X|\geq \sigma) \geq \mathbb{P}(|X - \mu| \geq \sigma) \geq 3 / 10$:
    \begin{equation}
        \mathbb{P} \left( \left| \langle z_{i+1}, \mu \rangle \right| \geq \frac{k^{-1/4}}{2(\log(k)+1)} \middle| \mathcal{F}_i,K_{i+1} = \hat{K}_{i + 1} \right) \geq \frac{3}{10}.
        \label{eq:310_case1}
    \end{equation}
    }

    \revise{
    \textbf{Case 2}: \( K_{i+1} = 0 \), then we still have \( \| K_{i+1} \| = 0 \leq \log(k) \), and following the same argument from Case 1, we have
    \begin{equation}
        \mathbb{P} \left( \left| \langle z_{i+1}, \mu \rangle \right| \geq \frac{k^{-1/4}}{2(\log(k)+1)} \middle| \mathcal{F}_i,K_{i+1} = 0 \right) \geq \frac{3}{10}.
        \label{eq:310_case2}
    \end{equation}
    }

    \revise{
    The conclusion~\eqref{eq:310} follows from combining~\eqref{eq:310_case1} and~\eqref{eq:310_case2}.
    }
\end{proof}

An upper bound on $\| V_k^{-1} \|$ can be obtained by applying an anti-concentration property of BMSB, along with a covering argument in~\cite{simchowitz2018learning}:

\begin{lemma}
    For any fixed $k \geq 3$ and $0 < \delta_k \leq 1/2$, it holds
    \begin{align}
        & \mathbb{P}\left( \left\| V_k^{-1} \right\|  \geq   \frac{12800}{3}k^{-1/2}(\log(k))^2 \right) \leq \nonumber \\
        & \quad \delta_k + e^{-\frac{3}{800}k + (m+n) \log(38400C_Vk^{1/2}(\log(k/\delta_k))^4) },
        \label{eq:min_eig}
    \end{align}
    where $C_V$ is defined in Theorem~\ref{thm:prop}, item~\ref{item:Eest}).
    \label{lemma:min_eig}
\end{lemma}

\begin{proof}
    Firstly, for any fixed $\mu \in \mathbb{S}^{m+n}$, applying~\cite[Prop. 2.5]{simchowitz2018learning} to the process $\{\langle z_i, \mu \rangle\}_{i=1}^{k-2}$, which is $\left(1, \frac{k^{-1/4}}{2(\log(k)+1)} , 3/10\right)$-BMSB by Lemma~\ref{lemma:verify_bmsb}, we have
    \begin{equation*}
        \mathbb{P}\left( \sum_{i=1}^{k-2} \langle z_i, \mu \rangle ^2 \leq \frac{k^{-1/2}(3/10)^2}{32(\log(k)+1)^2} (k-2) \right) \leq e^{-\frac{(3/10)^2(k-2)}{8}},
    \end{equation*}
    which implies
    \begin{equation}
        \mathbb{P}\left( \mu^\top V_k \mu \leq \frac{3}{12800}\cdot\frac{k^{1/2}}{(\log(k))^2} \right) \leq e^{-\frac{3}{800}k}
        \label{eq:muTVmu}
    \end{equation}
    by some simple scaling operations.
    Next, we shall choose multiple $\mu$'s and using a covering argument to lower bound the minimum eigenvalue of $V_k$, and hence to upper bound $\| V_k^{-1} \|$: let $\overline{\Gamma} = C_V k (\log(k / \delta_k))^2 I_{m+n}$, and $\underline{\Gamma} = (3/12800)k^{1/2}(\log(k))^{-2}I_{m+n}$. Let $\mathcal{T}$ be a minimal $1/4$-net of $\mathcal{S}_{\underline{\Gamma}}$ in the norm $\| \cdot \|_{\overline{\Gamma}}$\revise{, where $\mathbb{S}_{\underline{\Gamma}}$ is defined as the unit sphere with scaling matrix $\underline{\Gamma}$, i.e., $\mathbb{S}_{\underline{\Gamma}} = \{ v \in \mathbb{R}^{m+n} | v^\top \underline{\Gamma}^{-1} v = 1   \}$.}
    By \cite[Lemma~D.1]{simchowitz2018learning}, we have
    \begin{align}
        \log \left| \mathcal{T} \right| & \leq (m+n) \log(9) + \log \det(\overline{\Gamma} \underline{\Gamma}^{-1}) \nonumber \\
        & \leq (m+n)\log(38400C_V k^{1/2}(\log(k / \delta_k))^4).
        \label{eq:Tsize}
    \end{align}
    According to~\eqref{eq:muTVmu} and~\eqref{eq:Tsize}, we have
    \begin{align}
        & \mathbb{P}\left( \mu^\top V_k \mu \leq\frac{3}{12800}\cdot\frac{k^{1/2}}{(\log(k))^2} ,\;\forall \mu \in \mathcal{T} \right) \leq | \mathcal{T} | e^{-\frac{3}{800}k} \nonumber \\
        & \leq e^{-\frac{3}{800}k + (m+n) \log(38400 C_V k^{1/2}(\log(k/\delta_k))^4) }.
        \label{eq:Pmin}
    \end{align}
    On the other hand, according to Proposition~\ref{prop:max_eig} and Theorem~\ref{thm:prop}, item~\ref{item:Enoise}), we have
    \begin{equation}
         \mathbb{P} \left(\| V_k \| \leq C_V k (\log(k / \delta_k))^2  \right) \leq \delta_k,
        \label{eq:Pmax}
    \end{equation}
    From~\eqref{eq:Pmin}, \eqref{eq:Pmax} and \cite[Lemma~D.1]{simchowitz2018learning}, it follows that $\mathbb{P} \left( V_k \succ \underline{\Gamma} / 2 \right)$ is no greater than the RHS of~\eqref{eq:min_eig}, which is equivalent to the conclusion.
\end{proof}

The next proposition converts Lemma~\ref{lemma:min_eig} into a time-uniform bound:

\begin{proposition}
    For $\delta$ satisfying~\eqref{eq:delta_range} and $k_0$ defined in~\eqref{eq:k0}, it holds
    \begin{equation*}
        \mathbb{P} \left( \left\| V_{k}^{-1} \right\| \leq \frac{12800}{3} k^{-1/2}(\log(k))^2, \forall k \geq k_0 \right) \leq 3\delta.
    \end{equation*}
    \label{prop:min_eig}
\end{proposition}

\begin{proof}
    For $k\geq k_0$, with $\delta_k = \delta / k^2$, it holds
    \begin{equation*}
        k \geq \left( 5000 (m+n) \right)^{4/3},
    \end{equation*}
    and hence,
    \begin{align}
        & (m+n) \log(38400 C_V k^{1/2}(\log(k/\delta_k))^4) \nonumber \\
        \leq & (m+n) \left( 5\log(1 / \delta) + \frac{25}{2} \log(k) \right) \nonumber \\
        \leq & \frac{1}{800}k + \frac{25}{2}(m+n)k^{1/4} \leq \frac{1}{400}k.
        \label{eq:min_eig_simplify}
    \end{align}
    Substituting~\eqref{eq:min_eig_simplify} into~\eqref{eq:min_eig}, we have
    \begin{equation*}
        \mathbb{P}\left(\left\| V_k^{-1}\right\|  \geq   \frac{12800}{3}k^{-1/2}(\log(k))^2 \right) \leq \frac{\delta}{k^2} + e^{-\frac{1}{800}k}.
    \end{equation*}
    Taking the union bound over $k = k_0, k_0+1,\ldots$, we have
    \begin{align*}
        & \mathbb{P}\left(\left\| V_k^{-1} \right\|  \geq   \frac{12800}{3}k^{-1/2}(\log(k))^2, \forall k\geq k_0 \right) \nonumber \\
        \leq & \frac{\pi^2}{6} \delta + 801 e^{-\frac{1}{800}k_0} \leq 3 \delta. \qedhere
    \end{align*}
\end{proof}

Theorem~\ref{thm:prop}, item~\ref{item:Eest}) follows from Propositions~\ref{prop:Els}, \ref{prop:max_eig} and \ref{prop:min_eig}.

\subsection{Proof of Theorem~\ref{thm:prop}, item~\ref{item:Tstab})}

\begin{proof}
Let
\begin{align}
    & T_{1} = \inf\left\{ T \vphantom{\left(A^{t_k}\right)^\top}  \right. \left| \left(A^{t_k}\right)^\top P^* A^{t_k} < \rho P^*,  \forall k \geq T \right\}, \\
    & T_{2} = \inf\left\{ T \vphantom{\left(A+B\hat{K}_k\right)^\top} \right.\left| \left(A+B\hat{K}_k\right)^\top P^* \left(A+B\hat{K}_k\right) < \rho P^*,\right. \nonumber\\
    & \left. \mkern100mu \vphantom{\left(A+B\hat{K}_k\right)^\top}\forall k \geq T \right\}.
\end{align}
We shall bound $T_1$ and $T_2$ respectively:

\begin{enumerate}
    \item From the assumption $\rho(A) < 1$, it holds
    \begin{equation*}
        \lim_{k\to\infty}\left(A^k\right)^\top P A^k = 0,
    \end{equation*}
    which, together with $t_k = \lfloor \log(k) \rfloor$, implies $T_1$ is a finite constant independent of $\delta$, i.e., $T_1 \lesssim 1$.

    \item Since $(A+B{K})^\top P^* (A+B{K})$ is a continuous function of $K$, from~\eqref{eq:rho_star}, there exists a system-dependent constant $\epsilon_K$, such that $(A+B{K})^\top P^* (A+B{K}) < \rho P^*$ whenever $\| K - K^* \| < \epsilon_K$. On the other hand, since $\hat{K}_k$ is a continuous function of $\hat{\Theta}_k$~\cite[Proposition~6]{simchowitz2020naive}, there exists a system-dependent constant $\epsilon_\Theta$, such that $\| \hat{K} - K^* \| < \epsilon_K$ as long as $\| \hat{\Theta}_k - \Theta \| < \epsilon_\Theta$.
    It follows from Theorem~\ref{thm:prop}, item~\ref{item:Eest}) that $\| \hat{\Theta}_k - \Theta \| < \epsilon_\Theta$ whenever $k\geq (9C_\Theta^2 / \epsilon_\Theta^2)(\log(1 / \delta))^2$, and hence $T_2 \lesssim (\log(1 / \delta))^2$.
\end{enumerate}

In summary, it holds $T_{\text{stab}} = \max\{T_1, T_2 \} \lesssim (\log(1 / \delta))^2$.
\end{proof}

\subsection{Proof of Theorem~\ref{thm:prop}, item~\ref{item:Tnocb})}

This subsection is dedicated to bounding the time after which the circuit-breaking is not triggered any more.
The outline of the proof is stated as follows: firstly, we define a subsequence notation to deal with the dwell time $t_k$ of circuit breaking.
Secondly, an upper bound on the state and the certainty equivalent input after $T_{\text{stab}}$ is derived, which is shown to be asymptotically smaller than the circuit-breaking threshold $M_k = \log(k)$.
Based on the above upper bound on the certainty equivalent input, we can finally bound $T_{\text{nocb}}$, i.e., the time it takes for the certainty equivalent input to stay below the threshold $M_k$, which leads to the desired conclusion.

\subsubsection{A subsequence notation}

Consider the subsequence of states and inputs, where steps within the circuit-breaking period are skipped, defined below:
\revise{\begin{align}
&i(1)=1, i(j+1)=\left\{\begin{array}{ll}
i(j)+1 & {u}^{cb}_{i(j)} \neq 0 \\
i(j)+t_{i(j)} & {u}^{cb}_{i(j)}=0
\end{array},\right. \label{eq:subseq1}\\
&\tilde{x}_{j}=x_{i(j)},\tilde{u}^{ce}_{j}=u^{ce}_{i(j)},\tilde{u}^{cb}_{j}=u^{cb}_{i(j)}. \label{eq:subseq2}
\end{align}
It follows that $\{ \tilde{x}_j \}$ evolves as $
\tilde{x}_{j+1} = \tilde{A}_j \tilde{x}_j +  \tilde{w}_j
$, where:
\begin{align}
    & \tilde{A}_j = \begin{cases}
        A + B \hat{K}_{i(j)} & u^{cb}_{i(j)} \neq 0 \\
        A^{t_{i(j)}} & u^{cb}_{i(j)} = 0,
    \end{cases} \label{eq:subseq3} \\
    & \tilde{w}_j = \begin{cases}
        w_{i(j)} & u^{cb}_{i(j)} \neq 0 \\
        \sum_{\tau = 0}^{t_{i(j)}  - 1} A^\tau w_{i(j-1)+\tau} & u^{cb}_{i(j)} = 0.
    \end{cases}
    \label{eq:subseq4}
\end{align}
}

Consider $T_{\text{stab}}$ defined in~\eqref{eq:Tstab}. We can define $\tilde{T}_{\text{stab}}$ as the first index in the above subsequence for which the stabilization condition is satisfied, i.e.,
\begin{equation}
    \tilde{T}_{\text{stab}} = \inf\{ T \mid i(T) \geq T_{\text{stab}} \}.
    \label{eq:Tstabt}
\end{equation}

\subsubsection{Upper bound on state and certainty equivalent input after $T_{\text{stab}}$}

\begin{proposition}
    On the event $\mathcal{E}_{\text{noise}}(\delta) \cap \mathcal{E}_{\text{est}}(\delta)$, it holds
    \begin{align}
        \left\| \tilde{x}_{\tilde{T}_{\text{stab}}+k} \right\|  \lesssim \rho^{k/2} \log(1 / \delta) + \sqrt{\log(k / \delta)}, \label{eq:bound_xt}\\
        \left\| \tilde{u}^{ce}_{\tilde{T}_{\text{stab}}+k} \right\|  \lesssim \rho^{k/2} \log(1 / \delta) + \sqrt{\log(k / \delta)}, \label{eq:bound_ut}
    \end{align}
    where $\rho = (1 + \rho^*) / 2$.
    \label{prop:u_upper_bound}
\end{proposition}

\begin{proof}
    We can expand $\tilde{x}_{\tilde{T}_{\mathrm{stab}}+k}$ as:
    \begin{align}
        \tilde{x}_{\tilde{T}_{\text{stab}} + k} =& \tilde{A}_{\tilde{T}_{\text{stab}}+k-1}\tilde{A}_{\tilde{T}_{\text{stab}}+k-2}\cdots\tilde{A}_{\tilde{T}_{\text{stab}}} \tilde{x}_{\tilde{T}_{\text{stab}}} + \nonumber\\
        & \tilde{A}_{\tilde{T}_{\text{stab}}+k-1}\tilde{A}_{\tilde{T}_{\text{stab}}+k-2}\cdots\tilde{A}_{\tilde{T}_{\text{stab}} + 1}\tilde{w}_{\tilde{T}_{\text{stab}}} + \nonumber\\
        & \cdots + \nonumber\\
        & \tilde{w}_{\tilde{T}_{\text{stab}}+k-1},
    \end{align}
    where:
    \begin{itemize}
        \item $\tilde{A}_j \in \{A+B\hat{K}_{i(j)}, A^{t_{i(j)}}\}$ \revise{by~\eqref{eq:subseq3}}, and must satisfy $\tilde{A}_j^\top P \tilde{A}_j < \rho P$ for $j \geq \tilde{T}_{\text{stab}}$, by definition of $\tilde{T}_{\text{stab}}$ in~\eqref{eq:Tstab};
        \item $\tilde{w}_j \in \{w_{i(j)}, \sum_{\tau = 0}^{t_{i(j)}  - 1} A^\tau w_{i(j-1)+\tau}\}$ \revise{by~\eqref{eq:subseq4}}, and on the event $\mathcal{E}_{\text{noise}}(\delta)$, it must satisfy $\|\tilde{w}_j\| \lesssim \mathcal{A}\sqrt{\log(j / \delta)} \lesssim \sqrt{\log(j / \delta)}$ for any $j$, where $\mathcal{A} = \sum_{\tau = 0}^\infty \| A^\tau \|$.
    \end{itemize}

    Combining the above two items, we have
    \begin{align}
        \left\| \tilde{x}_{\tilde{T}_{\text{stab}}+k} \right\|  \lesssim \rho^{k/2}
        \left\| \tilde{x}_{\tilde{T}_{\text{stab}}} \right\| + \sqrt{\log\left(i\left(\tilde{T}_{\text{stab}}+k\right) / \delta\right)}.
        \label{eq:state_norm_after_stab}
    \end{align}

    We shall next bound $\left\| \tilde{x}_{\tilde{T}_{\mathrm{stab}}} \right\|$ and $i(\tilde{T}_{\text{stab}}+k)$ respectively:

    \begin{enumerate}
        \item
        According to the definition of $\tilde{T}_{\text{stab}}$ and $i(\cdot)$ in~\eqref{eq:subseq1} and~\eqref{eq:Tstabt}, we have
        \begin{equation}
            i\left(\tilde{T}_{\text{stab}}\right) \leq T_{\text{stab}} + \log(T_{\text{stab}}).
            \label{eq:i_base}
        \end{equation}
        On $\mathcal{E}_{\text{est}}(\delta)$, according to Theorem~\ref{thm:prop}, item~\ref{item:Tstab}, we have $T_{\text{stab}} \lesssim (\log(1 / \delta))^2$, and hence
        \begin{align}
            & i\left(\tilde{T}_{\text{stab}}\right) \lesssim (\log(1 / \delta))^2 + \log((\log(1 / \delta))^2) \nonumber \\
            & \lesssim (\log(1 / \delta))^2.
        \end{align}
        Furthermore, according to Theorem~\ref{thm:prop}, item~\ref{item:xnorm}, on $\mathcal{E}_{\text{noise}}(\delta)$, we have $\| x_k \| \lesssim \log(k / \delta)$ for any $k$, and hence
        \begin{align}
            & \left\| \tilde{x}_{\tilde{T}_{\text{stab}}} \right\| = \left\| {x}_{i\left(\tilde{T}_{\text{stab}}\right)} \right\|  \lesssim \log\left(\frac{(\log(1 / \delta))^2}{\delta}\right) \nonumber\\
            &= \log(1 / \delta) + \log((\log(1 / \delta))^2) \lesssim \log(1 / \delta).
            \label{eq:xtTstab}
        \end{align}

        \item By definition of $i(\cdot)$ in~\eqref{eq:subseq1}, we have
        \begin{equation}
            i \left( \tilde{T}_{\text{stab}} + k + 1 \right) \leq \tilde{T}_{\text{stab}} + k + \log \left( \tilde{T}_{\text{stab}} + k \right), \forall k \in \mathbb{N}^*.
            \label{eq:i_rec}
        \end{equation}
        Applying induction to~\eqref{eq:i_base} and~\eqref{eq:i_rec}, we can obtain
        \begin{equation}
            i\left(\tilde{T}_{\text{stab}}+k\right) \lesssim k\log(T_{\text{stab}}k) + T_{\text{stab}}.
            \label{eq:iTstabk1}
        \end{equation}
        Substituting $T_{\text{stab}} \lesssim (\log(\revise{1} / \delta))^2$, which holds on $\mathcal{E}_{\text{est}}(\delta)$ according to Theorem~\ref{thm:prop}, item~\ref{item:Tstab}), into~\eqref{eq:iTstabk1}, we have
        \begin{equation}
            i\left(\tilde{T}_{\text{stab}}+k\right) \lesssim k(\log(k / \delta))^2.
            \label{eq:iTstabk2}
        \end{equation}
    \end{enumerate}

    Hence, inequality~\eqref{eq:bound_xt} follows from substituting~\eqref{eq:xtTstab} and \eqref{eq:iTstabk2} into~\eqref{eq:state_norm_after_stab}. Moreover, since $\left\| \hat{K}_{i(\tilde{T}_{\text{stab}}+k)}\right\|$ is uniformly bounded by definition of $T_{\text{stab}}$ in~\eqref{eq:Tstab}, we have
    \begin{equation*}
        \left\| \tilde{u}^{ce}_{\tilde{T}_{\text{stab}}+k} \right\| \leq\left\| \hat{K}_{i(\tilde{T}_{\text{stab}}+k)}\right\| \left\| \tilde{x}_{\tilde{T}_{\text{stab}}+k} \right\| \lesssim \left\| \tilde{x}_{\tilde{T}_{\text{stab}}+k} \right\|,
    \end{equation*}
    which implies~\eqref{eq:bound_ut}.
\end{proof}

\subsubsection{Upper bound on \( T_{\text{nocb}} \)}

Now we are ready to prove Theorem~\ref{thm:prop}, item~\ref{item:Tnocb}):

\begin{proof}
    \revise{Consider a fixed $\alpha > 0$. For any arbitrarily small \( \epsilon > 0 \)}, according to~\eqref{eq:iTstabk2}, we have \( i(\tilde{T}_{\text{stab}} + k) \lesssim (1 / \delta)^{\alpha+\epsilon} \) as long as \( k \lesssim (1 / \delta)^{\alpha - \epsilon} \). Hence, \revise{by definition of $T_{\text{nocb}}$ in~\eqref{eq:Tnocb}}, we only need to prove
    \begin{equation}
        u_k^{cb} \revise{=} u_k^{ce} , \quad \forall k \geq i\left(\tilde{T}_{\text{stab}} + k_0\right),
        \label{eq:u_equiv}
    \end{equation}
    for some \( k_0 \lesssim (1 / \delta)^{\alpha - \epsilon} \).

    Notice that~\eqref{eq:u_equiv} is equivalent to
    \begin{equation*}
        \tilde{u}^{cb}_{{\tilde{T}_{\text{stab}}+k}} \revise{=} \tilde{u}^{ce}_{{\tilde{T}_{\text{stab}}+k}}, \quad \forall k \geq k_0,
    \end{equation*}
    which is then equivalent to
    \begin{equation*}
        \max\left\{\left\| \tilde{u}^{ce}_{\tilde{T}_{\text{stab}}+k} \right\|, \left\| \hat{K}_{\tilde{T}_{i(\text{stab})}+k} \right\|\right\} \leq M_k = \log(k), \quad \forall k \geq k_0.
    \end{equation*}
    Since $\left\| \hat{K}_{i(\tilde{T}_{\text{stab}}+k)}\right\|$ is uniformly bounded by definition of $T_{\text{stab}}$ in~\eqref{eq:Tstab}, according to Proposition~\ref{prop:u_upper_bound}, we only need to verify
    \begin{equation*}
        \rho^{k/2} \log(1 / \delta) + \sqrt{\log(k / \delta)}  \lesssim \log(k),
    \end{equation*}
    whenever \( k \gtrsim (1 / \delta)^{\alpha - \epsilon} \). In such case, we have
    \begin{align*}
        & \rho^{k/2} \log(1 / \delta) + \sqrt{\log(k / \delta)} \nonumber \\
        & \lesssim \log(1 / \delta) + \sqrt{\log(k) + \log(1 / \delta)} \nonumber \\
        & \lesssim \log(k) + \sqrt{\log(k)} \lesssim \log(k),
    \end{align*}
    from which the conclusion follows.
\end{proof}

\section{Proof of Theorem~\ref{thm:R_prob}}
\label{app:proof_R_prob}

In this appendix, we first decompose the regret of the proposed controller into multiple terms, then derive upper bounds on the terms respectively to obtain the high-probability regret bound stated in Theorem~\ref{thm:R_prob}.

\subsection{Regret decomposition}

\begin{lemma}
    Let \( K^*, P^* \) be defined in~\eqref{eq:K}, \eqref{eq:dare} respectively, and \(K = K^* + \Delta K\), then
    \begin{align}
        & Q + {K}^\top R{K} + \left(A+B{K}\right)^\top P^*\left(A+B{K}\right) - P^*  \nonumber \\
        =& \Delta K^\top(R+B^\top P^*B) \Delta K.
        \label{eq:perturb_K}
    \end{align}
    \label{lemma:perturb_K}
\end{lemma}

\begin{proof}
    Substituting the Lyapunov equation~\eqref{eq:lyap_opt} into the LHS of~\eqref{eq:perturb_K}, we have
    \begin{align*}
        &Q + {K}^\top R{K} + (A+B{K})^\top P^*(A+B{K}) - P^*  \nonumber \\
        =& {K}^\top R{K} - \left({K^*}\right)^\top RK^* + (A+B{K})^\top P^*(A+B{K}) - \nonumber \\
        &\quad \left(A+BK^*\right)^\top P^* \left(A+BK^*\right) \nonumber \\
        =& \Delta K^\top \left(R+B^\top P^* B\right) \Delta K + G + G^\top,
    \end{align*}
    where
    \begin{equation*}
        G = \Delta K^\top \left[(R+B^\top P^* B) K^* + B^\top P^* A\right].
    \end{equation*}
    By definition of \( K^* \) in~\eqref{eq:K}, we have \(G = 0\), and hence the conclusion holds.
\end{proof}

\begin{proposition}
    \revise{
    The regret of the proposed controller defined in~\eqref{eq:regret} can be decomposed as
    \begin{equation}
        \mathcal{R}(T) = \sum_{i=1}^5 \mathcal{R}_i(T),
        \label{eq:regret_decomposition}
    \end{equation}
    with the terms $\mathcal{R}_i(T)$ defined as:
    \begin{align}
        & \mathcal{R}_1(T) = \sum_{k=1}^T x_k^\top (K_k - K^*)^\top (R + B^\top P^*B)(K_k -K^*)x_k, \label{eq:R1} \\
        & \mathcal{R}_2(T) = 2\sum_{k=1}^T s_k^\top P^* (A+BK_k) x_k + (u_k^{pr})^\top R u_k^{cb} + \nonumber \\
        & \mkern70mu w_k^\top P^*B u_k^{pr},\\
        & \mathcal{R}_3(T) = \sum_{k=1}^T (u_k^{pr})^\top (R+B^\top P^* B) u_k^{pr}, \\
        & \mathcal{R}_4(T) = \sum_{k=1}^T w_k^\top P^*w_k  - TJ^*, \\
        & \mathcal{R}_5(T) = x_1^\top P^* x_1 - x_{T+1}^\top P^* x_{T+1}, \label{eq:R5}
    \end{align}
    where \(K_k, s_k\) are defined as:
    \begin{equation}
        K_k = \begin{cases}
            \hat{K}_k & u^{cb}_k = u^{ce}_k \\ 0 & \text{otherwise}
        \end{cases}, s_k = Bu^{pr}_k + w_k.
        \label{eq:Ks_def}
    \end{equation}
    }
    \label{prop:regret_decomposition}
\end{proposition}

\begin{proof}
    From $u_k = u_k^{cb} + u_k^{pr} = K_k x_k + u_k^{pr}$, it holds
    \begin{align}
        \mathcal{R}(T) &= \sum_{k=1}^T \left(x_k^\top Q x_k + u_k^\top R u_k\right) - TJ^* \nonumber\\
        &= \sum_{k=1}^T \left[x_k^\top\left(Q+K_k^\top RK_k\right)x_k + 2\left(u^{pr}_k\right)^\top Ru^{cb}_k+ \right. \nonumber\\
        &\qquad \left. \left(u^{pr}_k\right)^\top Ru^{pr}_k\right] - TJ^* \nonumber \\
        &= \sum_{k=1}^T \left[x_k^\top\left(Q+K_k^\top RK_k\right)x_k + x_{k+1}^\top P^*x_{k+1} - \right.\nonumber
        \\ & \qquad \left.x_k^\top P^*x_k\right] -TJ^* +  \nonumber \\
        & \qquad 2\sum_{k=1}^T (u_k^{pr})^\top  R u_k^{cb} +\sum_{k=1}^T (u_k^{pr})^\top R u_k^{pr} + \nonumber \\
        & \qquad \mathcal{R}_5(T).
    \end{align}
    We can further expand the first term in the RHS of the above equality: from $x_{k+1} = (A+BK_k)x_k + s_k$, we have
    \begin{align}
        & x_k^\top\left(Q+K_k^\top RK_k\right)x_k + x_{k+1}^\top P^*x_{k+1} - x_k^\top P^*x_k \nonumber\\
        =& x_k^\top \left[ Q+K_k^\top RK_k + (A+BK_k)^\top P^*(A+BK_k) - P^* \right]x_k \nonumber\\
        & \quad + 2 s_k^\top P^*(A+BK_k) x_k+ s_k^\top P^*s_k \nonumber \\
        =& x_k^\top (K_k - K^*)^\top (R + B^\top P^*B)(K_k -K^*) x_k + \nonumber\\
        & \quad  2 s_k^\top P^*(A+BK_k) x_k+ s_k^\top P^*s_k,
    \end{align}
    where the last equality follows from Lemma~\ref{lemma:perturb_K}. It follows from simple algebra that
    \begin{align}
        & \sum_{k=1}^T\left[x_k^\top\left(Q+K_k^\top RK_k\right)x_k + x_{k+1}^\top P^*x_{k+1} - x_k^\top P^*x_k\right]
        \nonumber \\
        & \qquad -TJ^* + 2\sum_{k=1}^T (u_k^{pr})^\top  R u_k^{cb} +\sum_{k=1}^T (u_k^{pr})^\top R u_k^{pr} \nonumber \\
        & = \sum_{i=1}^4 \mathcal{R}_i(T),
    \end{align}
    and hence the conclusion follows.
\end{proof}

\subsection{Upper bound on regret terms}

Next we shall bound the terms $\mathcal{R}_i(T)\,(i=1,\ldots,5)$ respectively:

\begin{proposition}
    The regret terms defined in~\eqref{eq:R1}-\eqref{eq:R5} can be bounded as follows:
    \begin{enumerate}
        \item On the event \( \mathcal{E}_{\text{noise}}(\delta) \cap \mathcal{E}_{\text{est}}(\delta) \), for $T > T_{\text{nocb}}$, it holds
        \begin{equation}
            \mathcal{R}_1(T) \lesssim (1 / \delta)^{1/4}(\log(1 / \delta))^4 + \sqrt{T}(\log(T / \delta))^5.
            \label{eq:bnd_R1}
        \end{equation}

        \item On the event \( \mathcal{E}_{\text{cross}}(\delta) \), it holds
        \begin{equation}
            \left| \mathcal{R}_2(T) \right| \lesssim \sqrt{T} (\log(T / \delta))^2.
            \label{eq:bnd_R2}
        \end{equation}

        \item On the event  \( \mathcal{E}_{\text{noise}}(\delta)\), it holds
        \begin{equation}
            |\mathcal{R}_3(T)| \lesssim \sqrt{T}\log(T / \delta).
            \label{eq:bnd_R3}
        \end{equation}

        \item On the event \( \mathcal{E}_{\text{cov}} (\delta)\), it holds
        \begin{equation}
            | \mathcal{R}_4(T) | \revise{\lesssim} \sqrt{T \log(1 / \delta)}.
            \label{eq:bnd_R4}
        \end{equation}

        \item On the event  \( \mathcal{E}_{\text{noise}}(\delta) \), it holds
        \revise{
        \begin{equation}
            | \mathcal{R}_5(T) | \lesssim (\log(T / \delta))^2.
            \label{eq:bnd_R5}
        \end{equation}
        }
    \end{enumerate}
    \label{prop:regret_terms}
\end{proposition}

\begin{proof}
    \begin{enumerate}
        \item
        Let
        \begin{equation}
            r_{1k} = x_k^\top (K_k - K^*)^\top \left(R + B^\top P^*B\right)(K_k -K^*)x_k.
        \end{equation}
        We shall next bound $\sum_{k=1}^{T_{\text{nocb}}} r_{1k}$ and $\sum_{k = T_{\text{nocb} + 1}}^T r_{1k}$ respectively:
        \begin{enumerate}
            \item For $k\leq T_{\text{nocb}}$, we have $\| x_k \| \lesssim \log( k / \delta)$ by Theorem~\ref{thm:prop}, item~\ref{item:xnorm}), and $\| K_k x_k \| = \| u_k^{cb} \| \leq \log(k)$. Therefore, it holds
            \begin{equation*}
                r_{1k} \lesssim (\log(k / \delta))^2,
            \end{equation*}
            and hence,
            \begin{equation*}
                \sum_{k=1}^{T_{\text{nocb}}} r_{1k} \lesssim T_{\text{nocb}} (\log(T_{\text{nocb}}/\delta))^2.
            \end{equation*}
            Invoking Theorem~\ref{thm:prop}, item~\ref{item:Tnocb}) with $\alpha = 1/4$, we get
            \begin{equation}
                \sum_{k=1}^{T_{\text{nocb}}} r_{1k} \lesssim (1 / \delta)^{1/4}(\log(1 / \delta))^4.
                \label{eq:r11}
            \end{equation}

            \item For \(k\leq T_{\text{nocb}}\), by definition of \(T_{\text{nocb}}\), we have \(K_k = \hat{K}_k\).
            \revisethree{On the event $\mathcal{E}_{\text{est}}$ defined in Theorem~\ref{thm:prop}, item~\ref{item:Eest}), $(\hat{A}_k, \hat{B}_k)$ converges to $(A, B)$, regardless of the controllability or stabilizability of $(\hat{A}_k, \hat{B}_k)$ at any particular $k$. Since $(A,B)$ is controllable by assumption and controllability is an open property~\cite{lewis2009semicontinuity}, we have that $(\hat{A}_k, \hat{B}_k)$ is controllable for sufficiently large $k$ on $\mathcal{E}_{\text{est}}$.}
            Hence, by definition of \( \mathcal{E}_{\text{est}} \) and the fact that \(\hat{K}_k\) is a continuous function of \(\hat{\Theta}_k\)~\cite[Proposition~6]{simchowitz2020naive}, we have
            \begin{align*}
                & \left\| K_k - K^* \right\| = \left\| \hat{K}_k - K^* \right\| \nonumber \\
                & \lesssim \left\| \hat{\Theta}_k - \Theta \right\| \lesssim k^{-1/4} (\log(k / \delta))^{3/2}.
            \end{align*}
            Furthermore, by Theorem~\ref{thm:prop}, item~\ref{item:xnorm}), we have \( \| x_k \| \lesssim \log(k / \delta) \), and hence,
            \begin{equation*}
                r_{1k} \lesssim \| K_k - K ^*\|^2 \|  x_k \|^2 \lesssim k^{-1/2} (\log(k / \delta))^5.
            \end{equation*}
            Therefore,
            \begin{equation}
                \sum_{k = T_{\text{nocb}} + 1}^T r_{1k} \lesssim \sqrt{T}(\log(T / \delta))^5.
                \label{eq:r12}
            \end{equation}
        \end{enumerate}
        Summing up~\eqref{eq:r11} and~\eqref{eq:r12} leads to~\eqref{eq:bnd_R1}.

        \item The inequality~\eqref{eq:bnd_R2} follows directly from the definition of \( \mathcal{E}_{\text{cross}} (\delta) \) in~\eqref{eq:Ecross}.

        \item \revise{From \( u_k^{pr} = k^{-1/4} v_k \), we have
        \begin{equation*}
            (u_k^{pr})^\top (R+B^\top P^* B) u_k^{pr} \lesssim k^{-1/2} \| v_k \|^2,
        \end{equation*}
        and hence, by definition of \( \mathcal{E}_{\text{noise}} \),
        \begin{equation*}
            (u_k^{pr})^\top (R+B^\top P^* B) u_k^{pr} \lesssim k^{-1/2} \log(k / \delta),
        \end{equation*}
        from which~\eqref{eq:bnd_R3} follows.}

        \item Since \( J^* = \tr(P^*) \) (see~\eqref{eq:Jstar}), we have
        \begin{align*}
            \left|\mathcal{R}_4(T)\right| &= \left|\sum_{k=1}^T \tr(w_kw_k^\top P) - T\tr(P)\right| \nonumber\\
            &= \left|\tr \left( \left( \sum_{k=1}^T(w_kw_k^\top - I_n) \right)P \right)\right| \nonumber\\
            &\lesssim \left\|  \sum_{k=1}^T(w_kw_k^\top - I_n) \right\| \lesssim \sqrt{T \log( 1 / \delta)},
        \end{align*}
        where the last inequality follows from the definition of \( \mathcal{E}_{\text{cov}}(\delta) \), which proves~\eqref{eq:bnd_R4}.

        \item \revise{The inequality~\eqref{eq:bnd_R5} is a direct corollary of Theorem~\ref{thm:prop}, item~\ref{item:xnorm}).}
        \qedhere
    \end{enumerate}
\end{proof}

Theorem~\ref{thm:R_prob} follows from combining Propositions~\ref{prop:regret_decomposition} and \ref{prop:regret_terms}.

\ifthenelse{\boolean{tacVersion}}{

}
{
  \section{Supplementary Note: Absence of Dwell Time May Lead to Destabilization}
  
In this note, we provide a counterexample to show that the absence of a dwell time $t_k$ in the proposed controller may lead to the destabilization of the system.

To facilitate the analysis, we consider a simplified setting where $\hat{K}$ is already the optimal controller, and no online adaptation is performed, i.e., $v_k = 0$ and $\hat{K}_k = K^*$ for any $k$. We focus instead on the stability of the system under the switching associated with the circuit-breaking mechanism. In particular, consider the switched system with modes
\begin{equation}
A_0 = \begin{bmatrix}
    0.9 & 0 \\
    0 & 0
\end{bmatrix}, A_1 = \begin{bmatrix}
    1.5 & 0.5 \\
    -2 & -0.5
\end{bmatrix},\footnote{
    These closed-loop system matrices can be realized by, e.g.,
    $A = \begin{bmatrix}
        1.5015 & 0.5006 \\
        -1.9997 & -0.4999
    \end{bmatrix}$, $B = 0.01I$, $Q=I$, $R=I$, $K^* = \begin{bmatrix}
        -0.1477 & -0.0594 \\
        -0.0260 & -0.0141
    \end{bmatrix}$, $K_0 = \begin{bmatrix}
        -60.1480 & -50.0594 \\
        199.974 & 49.9859
    \end{bmatrix}$. It can be verified that $A_0 = A+BK_0, A_1 = A+BK^*$, and that $K^*$ is the optimal LQR gain for $A, B, Q, R$. Here $K_0$ stands for the pre-stabilizer mentioned in Remark~\ref*{rem:prestabilizer}.
}
\label{eq:counterexample_A}
\end{equation}
both being stable since $\rho(A_0) = 0.9 < 1$ and $\rho(A_1) = 0.5 < 1$.
The system evolves, in the absence of dwell time, under the law
\begin{equation}
    \begin{cases}
        x_{k+1} = A_0 x_k + w_k, & \text{if } \| x_k \| \geq M_k, \\
        x_{k+1} = A_1 x_k + w_k, & \text{if } \| x_k \| < M_k,
    \end{cases}
    \label{eq:counterexample_system}
\end{equation}
where $w_k \stackrel{\text { i.i.d. }}{\sim} \mathcal{N}(0, I)$, and the thresholds $M_k$ be deterministic values satisfying
\begin{equation}
    M_k \geq (\log(k))^{\frac12 + \epsilon}, \limsup_{k\to\infty}\frac{M_{k+1}}{M_k} < 1.1,
    \label{eq:counterexample_M}
\end{equation}
where $\epsilon > 0$.

\begin{proposition}
    For the switched system described in \eqref{eq:counterexample_system} and \eqref{eq:counterexample_M}, it holds $\mathbb{P}(\lim_{k\to\infty} \| x_k \| = \infty) > 0$, i.e., the system has a nonzero probability of destabilization.
    \label{prop:counterexample}
\end{proposition}

\begin{proof}
    Define $\bar{x}_k = x_k / M_k$, $\bar{w}_k = w_k / M_{k+1}$, $\bar{A}_{0,k} = (M_k / M_{k+1})A_0$, $\bar{A}_{1,k} = (M_k / M_{k+1})A_1$, then~\eqref{eq:counterexample_system} can be rewritten as:
    \begin{equation}
        \begin{cases}
            \bar{x}_{k+1} = \bar{A}_{0,k} \bar{x}_k + \bar{w}_k, & \text{if } \| \bar{x}_k \| \geq 1, \\
            \bar{x}_{k+1} = \bar{A}_{1,k} \bar{x}_k + \bar{w}_k, & \text{if } \| \bar{x}_k \| < 1.
        \end{cases}
        \label{eq:counterexample_system_bar}
    \end{equation}
    Define
    \begin{equation}
    T = \sup\{k \geq 0 | M_{k+1} / M_k \geq 1.1\},
    \label{eq:counterexample_T}
    \end{equation}
    then according to~\eqref{eq:counterexample_M}, $T$ is a deterministic integer such that $T < \infty$.

    Consider the following two events:
    \begin{align}
        & \mathcal{E}_1 = \left\{ \left|e_1^\top \bar{x}_T\right| \geq 1 \right\}, \\
        & \mathcal{E}_2 = \left\{ \| \bar{w}_k \| \leq 0.01, \forall k \geq T \right\},
    \end{align}
    where $e_1 = \begin{bmatrix}
        1 & 0
    \end{bmatrix}^\top, e_2 = \begin{bmatrix}
        0 & 1
    \end{bmatrix}^\top$ and standard unit vectors.

    We proceed by showing that: a) $\mathbb{P}(\mathcal{E}_1) > 0$, b) $\mathbb{P}(\mathcal{E}_2) > 0$. Since $\mathcal{E}_1$ and $\mathcal{E}_2$ are independent events, a) and b) imply that $\mathbb{P}(\mathcal{E}_1 \cap \mathcal{E}_2) > 0$. Then it suffices to show that c) $\lim_{k\to\infty} \| x_k \| = \infty$ on $\mathcal{E}_1 \cap \mathcal{E}_2$.

    \paragraph{$\mathbb{P}(\mathcal{E}_1) > 0$} Let $\mathcal{F}_k$ be the $\sigma$-algebra generated by $\{w_0, \ldots, w_k\}$, then $e_1^\top x_k \mid \mathcal{F}_{k-1} \sim \mathcal{N}(\mu_k, 1)$ for any $k$, where $\mu_k \in \mathcal{F}_k$. Since $\mathcal{N}(\mu_k, 1)$ is supported on the entire real line, it follows that $\mathbb{P}(\left|e_1^\top x_k\right| \geq M_k \mid \mathcal{F}_{k-1}) > 0$ for any $k$, and hence $\mathbb{P}(\left|e_1^\top x_k\right| \geq M_k) > 0$ for any $k$. Choosing $k = T$ and plugging in $\bar{x}_k = x_k / M_k$ yields $\mathbb{P}(\mathcal{E}_1) > 0$.

    \paragraph{$\mathbb{P}(\mathcal{E}_2) > 0$} By~\eqref{eq:counterexample_M}, for any $k$, we have $$\left\{ \| \bar{w}_k \| > 0.01 \right\} \subseteq \left\{ \| w_k \| > 0.01 (\log(k+1))^{1/2+\epsilon} \right\},$$
    and hence, by~(\ref*{eq:w_tail}), we have
    \begin{equation}
    \mathbb{P}\left( \| \tilde{w}_k \| > 0.01 \right) \leq 2^{n/2}(k+1)^{-c (\log(k+1))^\epsilon}.
    \label{eq:counterexample_w_tail}
    \end{equation}
    Since the RHS of \eqref{eq:counterexample_w_tail} is summable, we have:
    \begin{equation*}
    \mathbb{P}\left( \| \bar{w}_k \| \leq 0.01, \forall k\geq k_0 \right) > 0
    \end{equation*}
    for some finite and deterministic $k_0$. Since $\{ \| \bar{w}_k \| \leq 0.01 \}$ are independent for different $k$ and has nonzero probability of each $k$, we further have $\mathbb{P}(\mathcal{E}_2) > 0$.

    \paragraph{$\lim_{k\to\infty} \| x_k \| = \infty$ on $\mathcal{E}_1 \cap \mathcal{E}_2$} Since $\lim_{k\to\infty} M_k = \infty$ and $\bar{x}_k = x_k / M_k$, we only need to prove that $\liminf_{k\to\infty} \| \bar{x}_k \| \geq 1$. Since $| e_1^\top \bar{x}_T | \geq 1$, we only need to prove that $| e_1^\top \bar{x}_{k+j} | \geq 1$ for some finite $j$ as long as $| e_1^\top \bar{x}_k | \geq 1$ and $k\geq T$, i.e., the first entry of $\bar{x}_k$ exceeds 1 for infinitely many $k$'s.

    Now we assume that $k \geq T$ and $| e_1^\top \bar{x}_k | \geq 1$. Let $i = \inf\{ i \geq 0 | \| \bar{x}_{k+i} \| < 1  \}$, i.e., $k+i$ is the first time after $k$ that $\bar{x}$ enters the region $\| \bar{x} \| < 1$. Note that our choice of $A_0$ removes the second component of the state vector, from which it follows that $| e_1^\top \bar{x}_{k+i-1} | \geq 0.99$: if $i = 1$, then the inequality follows trivially from the assumption; otherwise, $i \geq 2$, and
    \begin{align*}
        1 & \leq \| \bar{x}_{k+i-1} \| = \| A_0 \bar{x}_{k+i-2} + \bar{w}_{k+i-2} \| \\
        & = \left \|  (e_1^\top \bar{x}_{k+i-1}) e_1 + (e_2^\top \bar{w}_{k+i-2})e_2 \right \| \\
        & \leq | e_1^\top \bar{x}_{k+i-1} | + | e_2^\top \bar{w}_{k+i-2} | \leq | e_1^\top \bar{x}_{k+i-1} | + 0.01,
    \end{align*}
    from which it follows that $| e_1^\top \bar{x}_{k+i-1} | \geq 0.99$. Hence, we have
    \begin{align*}
        | e_1^\top \bar{x}_{k+i} | & = | e_1^\top A_0 \bar{x}_{k+i-1} + e_1^\top \bar{w}_{k+i-1} | \\
        & \geq 0.9 \frac{M_{k+i-1}}{M_{k+i}} | e_1^\top \bar{x}_{k+i-1} | - | e_1^\top \bar{w}_{k+i-1} | \\
        & \geq 0.9 \times \frac{1}{1.1} \times 0.99 - 0.01 = 0.8.
    \end{align*}
    Therefore, with $j = i + 1$, we have:
    \begin{align*}
        | e_1^\top \bar{x}_{k+j} | & = | e_1^\top A_1 \bar{x}_{k+i} + e_1^\top \bar{w}_{k+i} | \\
        & \geq 1.5 \frac{M_{k+j-1}}{M_{k+j}} | e_1^\top \bar{x}_{k+i} | - | e_1^\top \bar{w}_{k+i} | \\
        & \geq 1.5 \times \frac{1}{1.1} \times 0.8 - 0.01 > 1.
    \end{align*}
    Hence, we have found the desired $j$.
\end{proof}

A visual illustration of the proof above is shown in Fig.~\ref{fig:counterexample_proof}.

\begin{figure}[!htbp]
    \centering
  \tikzsetnextfilename{figures/counterexample_proof.tex}%
  \begin{tikzpicture}[
    nodeStyle/.style={circle, fill, draw, inner sep=2pt},
    lineStyle/.style={-latex, thick}
]
    \draw[thick,-latex] (0,1) -- (7.5,1) node[anchor=north west] {$k$};
    \draw[thick,-latex] (0,1) -- (0,5.5) node[anchor=south] {$\| x_k \| / M_k$};

    \draw (0,2) -- (-0.2,2) node[anchor=east] {$0.8$};
    \draw (0,3.5) -- (-0.2,3.5) node[anchor=east] {$1$};

    \draw[dash dot] (0,3.5) -- (7,3.5);
    \draw[dash dot] (0,2) -- (7,2);

    \node[nodeStyle] (S) at (0.4,4.8) {};
    \node[nodeStyle] (A) at (1,2.5) {};
    \node[nodeStyle] (B) at (2,5) {};
    \node[nodeStyle] (C) at (2.5, 4.7) {};
    \node[nodeStyle] (D) at (3,4) {};
    \node[nodeStyle] (E) at (4,2.8) {};
    \node[nodeStyle] (F) at (5,4.8) {};
    \node[nodeStyle] (G) at (5.5,4.5) {};
    \node[fill=none, anchor=west] (H) at (6,4.2) {\Large \textbf{$\cdots$}};

    \draw[lineStyle, blue] (S) -- node [near start, blue, anchor=west] {$A_0$} (A);
    \draw[lineStyle, red] (A) -- node [midway, red, anchor=west] {$A_1$} (B);
    \draw[lineStyle, blue] (B) -- (C);
    \draw[lineStyle, blue] (C) -- (D);
    \draw[lineStyle, blue] (D) -- node [near start, blue, anchor=west] {$A_0$}(E);
    \draw[lineStyle, red] (E) -- node [midway, red, anchor=west] {$A_1$}(F);
    \draw[lineStyle, blue] (F) -- (G);
    \draw[lineStyle, blue] (G) -- (H);

    \draw[dashed] (S) -- (0.4, 1);
    \draw (0.4, 1) -- (0.4, 0.8) node[anchor=north] {$T$};
\end{tikzpicture}%

    \caption{Visual illustration of the proof of Proposition~\ref{prop:counterexample}.
    Starting from step $T$ defined in \eqref{eq:counterexample_T}, the state norm $\| x_k \|$ oscillates indefinitely around the threshold $M_k$, since at every time step when $\| x_k \|$ drops below $M_k$, it must fall in the interval $[0.8 M_k, M_k]$, in which case the mode $A_1$ causes $\| x_k \|$ to exceed $M_k$ again.}
    \label{fig:counterexample_proof}
\end{figure}

So far, we have proved that switching between the two stable modes in~\eqref{eq:counterexample_A} without a dwell time may lead to the destabilization of the system. By contrast, it is evident from the main results that the introduction of a dwell time $t_k = \lfloor \log(k) \rfloor$ can prevent the destabilization almost surely. This comparison is also illustrated by numerical simulation: in Fig.~\ref{fig:counterexample}, we consider this particular switched system with thresholds $M_k = \log(k)$, and closed-loop systems with and without a dwell time are compared under the same realization of the process noise sequence sampled from $w_k \stackrel{\text { i.i.d. }}{\sim} \mathcal{N}(0, 0.01I)$. It is evident that the closed-loop system without a dwell time~\eqref{eq:counterexample_system}, has the state norm oscillating around the threshold $M_k$, and hence diverging to infinity. This diverging pattern is consistent with the proof of Proposition~\ref{prop:counterexample}. By contrast, with a  $t_k = \lfloor \log(k) \rfloor$ dwell time, i.e., the mode $A_0$ is selected for $\lfloor \log(k) \rfloor$ consecutive time steps if it is triggered at time $k$, the system is stabilized.

\begin{figure}[!htbp]
    \centering
  \tikzsetnextfilename{figures/counterexample.tex}%
  \input{figures/counterexample.tex}%

    \caption{Comparison of the trajectories of the switched system with modes $A_0, A_1$ in~\eqref{eq:counterexample_A} with and without a dwell time. The time axis is in logarithmic scale.}
    \label{fig:counterexample}
\end{figure}

When considering the adaptive control setting, where the probing noise $v_k$ is present, and the controller $\hat{K}_k$ is updated over time, we suspect the absence of a dwell time to have a similar destabilizing effect for certain systems, although the rigorous analysis can be more involved. For example, reference~\cite{wang2021exact}, which uses an algorithm similar to the one proposed in this paper but without a dwell time, does not ensure the stability of the closed-loop system, since the exchange of matrix multiplication order implicit in the proof of \cite[Lemma 43]{wang2021exact} may not generally hold true.

}

\begin{IEEEbiography}[{\includegraphics[width=1in,height=1.25in,clip,keepaspectratio]{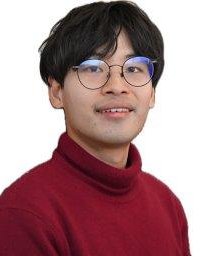}}]%
  {Yiwen Lu} received his Bachelor of Engineering degree from Department of Automation, Tsinghua University in 2020. He is currently a Ph.D. candidate in Department of Automation, Tsinghua University. His research interests include adaptive and learning-based control, with applications in robotics.
\end{IEEEbiography}
\vskip 0pt plus -1fil
\begin{IEEEbiography}[{\includegraphics[width=1in,height=1.25in,clip,keepaspectratio]{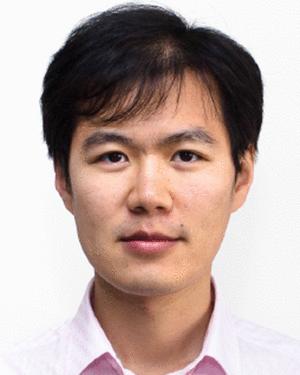}}]%
  {Yilin Mo} is an Associate Professor in the Department of Automation, Tsinghua University. He received his Ph.D. In Electrical and Computer Engineering from Carnegie Mellon University in 2012 and his Bachelor of Engineering degree from Department of Automation, Tsinghua University in 2007. Prior to his current position, he was a postdoctoral scholar at Carnegie Mellon University in 2013 and California Institute of Technology from 2013 to 2015. He held an assistant professor position in the School of Electrical and Electronic Engineering at Nanyang Technological University from 2015 to 2018. His research interests include secure control systems and networked control systems, with applications in sensor networks and power grids.
\end{IEEEbiography}

\end{document}